\documentclass[12pt]{amsart}

\usepackage[english]{babel}
\usepackage{graphicx,amsmath,amssymb,color, enumerate}
\numberwithin{equation}{section}

\usepackage{color}

\newtheorem{theorem}{Theorem}[section]
\newtheorem{lemma}[theorem]{Lemma}

\newtheorem{remark}[theorem]{Remark}

\newtheorem{remarks}[theorem]{Remarks}

\newcommand{\Bk}{\color{black}}

\newcommand{\Dom}{\mathsf{Dom}}

\newcommand{\Z}{\mathbb{Z}}
\newcommand{\R}{\mathbb{R}}

\begin{document}

\title[Semi-classical edge states for the Robin Laplacian]{ Semi-classical edge states for the Robin Laplacian}

\author{B.  Helffer}
\address[B. Helffer]{Laboratoire de Math\'ematiques Jean Leray, Universit\'e de Nantes, France.}
\email{Bernard.Helffer@univ-nantes.fr}

\author{A.  Kachmar}
\address[A. Kachmar]{Lebanese University, Department of Mathematics, Hadath, Lebanon.}
\email{akachmar@ul.edu.lb}

\date{\today}

\maketitle
\begin{abstract}
Motivated by the study of high energy Steklov eigenfunctions, we examine  the semi-classical Robin Laplacian. In the two dimensional situation, we determine an effective operator describing the asymptotic distribution of the negative  eigenvalues, and we prove that the corresponding eigenfunctions decay away from the boundary, for all dimensions. 
\end{abstract}

\section{Introduction}

\subsection{Motivation:  Generalized Steklov eigenfunctions}\

Let us consider an open bounded set $\Omega\subset\R^{\sf n}$ with a smooth connected boundary $\Gamma$. 
Let $-\Delta^D$ be the Dirichlet Laplace operator on $\Omega$ with spectrum $\sigma(-\Delta^D)$. We fix a constant $w\in\R\setminus\sigma(-\Delta^D)$.  For every function $\psi\in H^{1/2}(\Gamma)$, we assign the unique function $u=u_{w,\psi}$ as follows
\begin{equation}\label{eq:D}
-\Delta u=w \, u{~\rm on~}M\quad\mbox{ and } \quad u=\psi~{\rm on~}\Gamma\,.
\end{equation}
The operator  
\begin{equation}\label{eq:DtN}
 \psi\in H^{1/2}(\Gamma)\mapsto \Lambda(w) \psi:= \frac{\partial u_{w,\psi}}{\partial\nu}\in H^{-1/2}(\Gamma)
\end{equation}
is the Dirichlet to Neumann (DN) operator. Here $\nu$ denotes the unit outward normal vector of $\Gamma$. \\
The DN operator is a boundary pseudo-differential operator of order~$1$. Its spectrum consists of a non-decreasing sequence of eigenvalues $(\mu_m(w))_{m\geq1}$ counting multiplicities, known as the (generalized) Steklov eigenvalues\footnote{ The Steklov eigenvalues correspond to the case where $w=0$.}.  More precisely,
 \[\sigma(\Lambda (w))=\sigma^{\mathfrak s},\]
  where $\sigma^{\mathfrak s}$ is the Steklov spectrum
defined as the set of real numbers $\mu$ such that a non-trivial solution $u$ exists for the following Robin problem
\begin{equation}\label{eq:Robin}
-\Delta u=w\, u ~{\rm on~}\Omega\,,\quad u\in H^2(\Omega) ~\mbox{ and } ~\frac{\partial u}{\partial \nu} = \mu \,  u~{\rm on~}\Gamma\,.
\end{equation}
The study of  the localization of  the normalized solutions $u^\mu $ of \eqref{eq:Robin} in the limit\,\footnote{This amounts to the study of the Steklov eigenpair $(u^{\mu_m},\mu_m)$ as $m\to+\infty$.} $\mu\to+\infty$ is connected with  the semi-classical Robin Laplacian studied in \cite{HK-tams}.

Let us formulate the Steklov problem in the framework of \cite{HK-tams}. We introduce the semi-classical parameter $h=\mu^{-2}$ and  denote  by $u_h$ a non-trivial solution of \eqref{eq:Robin};  the eigenfunction $u_h$ satisfies 
\begin{equation}\label{eq:Robin*}
\begin{cases}
-\Delta u_h=w \,  u_h&~{\rm in~}\Omega\medskip\\
\displaystyle\frac{\partial u_h}{\partial \nu} = h^{-1/2} u&~{\rm on~}\Gamma
\end{cases}\,.
\end{equation}
We introduce the self-adjoint operator $\mathcal T_h$ with domain $D(\mathcal T_h)$ as follows
\[ \mathcal T_h=-h^2\Delta\,,\quad \mathfrak D(\mathcal T_h)=\{u\in H^2(\Omega)~:~\frac{\partial u}{\partial \nu} =h^{-1/2} u~{\rm on~}\Gamma\}\,.\]
Then \eqref{eq:Robin*} can be rewritten in the form
\begin{equation}\label{eq:Robin**} 
\mathcal T_hu_h=w_h u_h,\quad u_h\in \mathfrak D(\mathcal T_h)\setminus\{0\}\,,\quad w_h:=h^2w\,.
\end{equation}
By \cite{HK-tams},  in the planar situation $\mathsf n=2$,  if $w_h<0$ (see below the more precise condition), $u_h$ decays exponentially as follows: \\
{\it Given $M\in(0,1)$ and $\alpha\in(0,\sqrt{ M})$, there exist $h_0,C>0$ such that
\begin{equation}\label{eq:dec1}
\int_\Omega \left(|u_h|^2+h|\nabla u_h|^2 \right)\exp\left(\frac{2\alpha\,d(x,\Gamma)}{h^{1/2}}\right)\,dx \leq C\|u_h\|_{L^2(\Omega)}^2\,,  \end{equation}
for $h\in(0,h_0] $ and $w_h< - Mh$.}

 Here $d(\cdot,\Gamma)$ is the  normal distance to the boundary
\begin{equation}\label{eq:d(x,Gamma)} 
d(x,\Gamma)=\inf\{|x-y|~:~y\in\Gamma\}\quad(x\in\R^n)\,.
\end{equation}
This decay is a consequence of Agmon type estimates.   If we note that the ground state energy of the operator $\mathcal T_h$ satisfies 
 $\lambda_1(\mathcal T_h)=-h+o(h)$ as $h\to0_+$, the theorem applies with $\alpha <1$.  This decay result can be easily extended to the $\mathsf n$-dimensional situation \cite{PP}  from which we  can deduce pointwise estimates on $u_h$ (see Theorem~\ref{thm:dec-agm-pw}).

 Examining the case of the annulus, $\Omega= \{ x\in\R^2~:~r_0<|x|<1\}$, we observe that the constant $\alpha$ and the distance function $d(x,\Gamma)$ in \eqref{eq:dec1} are non-optimal. The example of the annulus suggest the optimal decay rate is achieved with $\alpha\approx 1$ and a distance function $ \hat d_{\Gamma}$  that depends on the curvature of the boundary (see \cite[Sec.~1.1.3]{GT}  and  \cite{DHN}).
 
Returning  to the problem in \eqref{eq:Robin}, we see  that  a consequence of \eqref{eq:dec1}  is that  the Steklov eigenfunction decays away from the boundary  provided the Steklov eigenvalue $\lambda$ satisfies  $w \leq -M \lambda^2$ and $\lambda\gg 1$ (i.e. $|w|\geq M\lambda^2\gg 1$).  

Our aim is to relax this strong assumption imposed on $w$.  
This question is motivated by the paper  by Galkowski-Toth \cite{GT} (who also refer to Hislop-Lutzer \cite{HiLu}  and Polterovich-Sher-Toth \cite{PST}) and by the PHD thesis  of G. Gendron \cite{GG}  discussing for special manifolds with boundary 
 the correspondence between the spectrum of the Steklov and the metric given on the manifold. In the first  contribution, it is assumed that $w =0$, and the above decay is obtained with $\alpha=1$, but under the condition that the boundary is analytic. Although not written explicitly, the computations by G. Gendron can also lead to the same result (but for a particular case).  This has been developed in  the recent work \cite{DHN}.

In the semi-classical framework, we will study the spectral properties of the eigenvalues of the Robin Laplacian $\mathcal T_h$ below the energy level $h^2\lambda_1^D(\Omega)$, where $\lambda_1^D(\Omega)$ is the ground state energy of the Dirichlet Laplacian. We obtain a boundary effective operator that describes the asymptotic distribution of the eigenvalues in the semi-classical limit (see Theorem~\ref{thm:main} below). The corresponding eigenfunctions (which can be viewed  as interior Steklov eigenfunctions in the sense of \cite{HiLu}  and  \cite{GT}) are expected to be localized near the domain's boundary (thereby called edge states in the literature). We confirm this property in Theorem~\ref{thm:conj} below, which is valid for any dimension ${\sf n}\geq 2$.

\subsection{Decay of eigenfunctions}
Using the boundary pseudo-differential calculus (as in \cite{HiLu}),  we obtain that all eigenfunctions corresponding to non-positive eigenvalues of the Robin Laplacian $\mathcal T_h$ decay away from the boundary, uniformly with respect to the non-positive eigenvalues.  This  extends the result of \cite{HiLu} up to the boundary, and presents a weaker version of the result of \cite{GT} but valid for the non-zero modes of $\mathcal T_h$. 

\begin{theorem}\label{thm:conj}
 Let $\lambda_1^D(\Omega)$ the principal eigenvalue of the Dirichlet Laplacian $-\Delta$ on $\Omega$.

For any $p\in \mathbb N$ and $ \epsilon < \lambda_1^D(\Omega)$,   there exist positive constants $C_{p,\epsilon},h_{p,\epsilon}$ and such that if $ (h,u_{h},w)$ is a solution of \eqref{eq:Robinw} 
\begin{equation}\label{eq:Robinw}
\begin{cases}
-\Delta u_h= w \, u_h &~{\rm in~}\Omega\medskip\\
\displaystyle\frac{\partial u_h}{\partial \nu} =h^{-1/2} u_h&~{\rm on~}\Gamma
\end{cases}\,,
\end{equation}
with $h\in (0,h_{p,\epsilon}]$, $w \leq  \epsilon$, and $\|u_h\|_{L^2(\partial\Omega)}=1$
 then it satisfies
 \begin{equation} \label{hk1.10} 
 |u_h (x)|\leq C_{p,\epsilon} \left(\frac{h}{\big(d(x,\Gamma)\big)^2}\right)^{p}\,,\, \forall x \in \Omega\,,
 \end{equation}
 where $d(x)=d(x,\Gamma)$ is the distance to the boundary introduced in \eqref{eq:d(x,Gamma)}.
 \end{theorem}
One could hope in the case of an analytic boundary to have by using an analytic pseudo-differential calculus a control of the constant $C_{p,\epsilon}$  in \eqref{hk1.10} with respect to $p$   leading  to  an estimate of the following form
     \begin{equation}\label{HLimprovedc}
 | u_h(x)| \leq  C_0 \, d(x)^{2-\mathsf n} h^{-\frac 12} \exp\left( - \frac 1{C_1}  d(x)h^{-\frac 12}\right)\,,\, \forall x \mbox{ s.t. } d(x) \leq \frac{1}{C_2}\,,
  \end{equation}
  for some constants $C_0, C_1, C_2>0$, which could be difficult to determine explicitly.   We will discuss this 
   in Subsection~\ref{rem:exp-dec}.  Note that, for $w=0$, \eqref{eq:GT-gen} is established with $C_1=1-\eta$ and $\eta$ arbitrarly small  in \cite{GT} by using analytic microlocal methods. This  was improving  the non-optimal exponential bound of \cite{PST} in the $2$D case.  
 
In the case of an analytic boundary, based on the  analysis in \cite{GT},  we are able to improve improve  the decay in Theorem~\ref{thm:conj} for $w\not=0$.
 
 \begin{theorem} \label{thm:conj-GT}
Assume that $\Gamma$, the boundary of  $\Omega$, is analytic. 
For any  $\zeta < \lambda_1^D(\Omega)$ and $\eta >0$,   there exist positive constants $\varepsilon,C, h_0$ such that if $ (h,u_{h},w)$ is a solution of \eqref{eq:Robinw} with  $h\in (0,h_{0}]$, $w \leq  \zeta$, and $\|u_h\|_{L^2(\partial\Omega)}=1$, then  the following estimate holds,
\begin{equation}\label{eq:GT-gena}
\forall\,x\in\Omega,~  |u_h(x)| \leq C \, h^{-\frac{n}4+\frac18} \exp\left(-\frac{(1-\eta)  \inf(d(x,\Gamma),\varepsilon) }{h^{1/2} }\right)\,.
\end{equation}
\end{theorem}
At the moment, it is unclear if the analytic assumptions are important for the validity of the estimates (and more accurate estimates discussed around  \eqref{eq:GT-gen}). 
 Note that  in the $C^\infty$ case, the microlocal approach proposed  in \cite{GT} could at most give an information modulo $\mathcal O (h^\infty)$  leading perhaps to \eqref{hk1.10} 
 with $\epsilon=\hat C h^\frac 12 |\log h|  $, for some $\hat C>0$, to compare with \eqref{hk1.10}.

 The method of Agmon estimates, recalled in \eqref{eq:dec1}, is on  one hand advantageous since it does not require the analytic hypothesis of the boundary, but on the other hand its drawback is that it becomes weaker, due to the condition on $\alpha$, as the  eigenvalue $w$ approaches $0$. However,  Theorem~1 of Galkowski-Toth \cite{GT}  and 
  Theorem~\ref{thm:conj} above show that  all eigenfunctions decay with a constant exponential profile under the  analytic boundary hypothesis. It would then be interesting to extend these estimates to the case of a $C^\infty$- boundary.

Positive indications will be given in the 2 dimensional case that we will discuss in the next section. Let us denote by
\begin{equation}\label{eq:def-L}
L=\frac{|\Gamma|}2\,,
\end{equation}
where  $|\Gamma|$ is the length of the boundary $\Gamma$. Assuming $\Gamma$ is connected, we will encounter quasi-modes normalized in $L^2(\Gamma)$ and having the following profile
\[u_h \approx (2L)^{-1/2}\exp\big(-h^{-1/2}d(x,\Gamma)\big) e^{ik\pi s/L} \]
with $k\in\Z$.
  Such quasi-modes appear also in Polterovich-Sher-Toth's paper \cite{PST}  for the eigenvalue $w=0$, where it is proved, in the case of an analytic boundary, that they are close to the actual zero-modes of  the operator $\mathcal T_h$. In the case where $\Gamma$ is not connected  \cite{PST}, we still encounter  the foregoing quasi-modes on each connected component of $\Gamma$ and their linear combinations.

\subsection{Asymptotic distribution of eigenvalues}

 There is a one-to-one correspondence between the negative eigenvalues of $\mathcal T_h$ and the Steklov  
eigenvalues below the energy level $h^{-1/2}$  (see \cite{BFK} and \cite[Lem.~1]{BCM} in a slightly different context). The correspondence being not explicit, it does not yield a precise description of the eigenvalues of the operator $\mathcal T_h$, based on the existing eigenvalue asymptotics for the Steklov eigenvalues, but it  does allow  to deduce the asymptotics for the counting function of the operator $\mathcal T_h$ from that of the DN operator $\Lambda(0)$. Our result on the Robin eigenvalues (Theorem~\ref{thm:main**}) is new and within our approach we can quantify the correspondence between the Robin and Steklov eigenvalues, and also to derive Weyl laws for the operator $\mathcal T_h$ (and consequently for the DN operator) in Theorem~\ref{thm:weyl}. 

Let us consider the case $\mathsf n=2$ for the sake of simplicity and assume that $\Omega$ is simply connected. We denote by $\big(\lambda_n(\mathcal T_h)\big)_{n\geq 1}$ the sequence of min-max eigenvalues of the operator $\mathcal T_h$. We will determine the asymptotic behavior of $\lambda_n(\mathcal T_h)$ in the regime $h \to 0_+$ thereby describing the distribution of all the negative eigenvalues of $\mathcal T_h$.

 For all $n\geq 2$ and $L$ introduced in \eqref{eq:def-L}, we introduce the eigenvalues
\[\lambda_n^{\rm F}(L)=\frac{\pi^2 k^2}{L^2}\quad{\rm for~}n\in\{2k,2k+1\}~\&~k\in\mathbb N\,,\]
which correspond to the  Fourier modes $e^{\pm i\pi ks/L}$ on $\R/ 2L\mathbb Z$.
\begin{theorem}\label{thm:main**}
Let  $\lambda_2^N(\Omega)$ denote the second eigenvalue of the Neumann Laplacian $-\Delta$ on $\Omega$ and consider a positive constant $\epsilon<\lambda_2^N(\Omega)$.  Assume furthermore that $\Omega$ is simply connected. Then, there exist positive constants $C$ and $h_0$ such that, for all $h\in(0,h_0] $, the following estimates hold,
\[ \big|\lambda_n(\mathcal T_h) +h-h^2\lambda_n^{\rm F}(L)\big|\leq Ch^{3/2}\big(1+h^{3/4}\lambda_n^{\rm F}(L)\big)\,,\]
provided that $\lambda_n(\mathcal T_h)<\epsilon h^2$\,.
\end{theorem}

The proof of Theorem~\ref{thm:main**} follows by deriving an effective operator approximating the operator $\mathcal T_h$. The precise statement is given in Theorem~\ref{thm:main}.

The estimates of Theorem~\ref{thm:main**} are interesting when $n\gg h^{-1/4}$, since by \cite[Thm.~2.1]{HK-tams} and \cite[Prop.~7.4]{HKR}, $\lambda_n(\mathcal T_h)\leq -h-h^{3/2}\kappa_{\max}+\mathcal O(h^{7/4})$ for $n\lesssim h^{-1/4}$. In particular, the negative eigenvalues of $\mathcal T_h$ satisfy
\[
\lambda_{2k}(\mathcal T_h)\sim -h+\pi^2L^{-2}k^2h^2~\mbox{ and }~\lambda_{2k+1}(\mathcal T_h)-\lambda_{2k}(\mathcal T_h) =\mathcal O(h^{3/2})\]
provided $k\gg h^{-1/4}$. This is consistent with the results in \cite{E, GiI, Roz} and \cite[Sec.~3.1]{PST} dealing with the spectrum of the DN operator $\Lambda(0)$, whose principal symbol coincides with $\sqrt{-\Delta_\Gamma}$, the square root of the Laplace-Beltrami operator on $\Gamma$.  In fact, the  Steklov eigenvalues $(\mu_n)_{n\geq 1}$ of $\Lambda(0)$ satisfy  the following asymptotics \cite{Roz}
\begin{equation}\label{asRoz}
 \mu_{2k+1}=\mu_{2k}+\mathcal O(k^{-\infty})=\frac{\pi}{L}k+\mathcal O(k^{-\infty})\quad(k\to+\infty)\,. 
 \end{equation}
So, we get  the following correspondence between the negative Robin eigenvalues $\{\lambda_n(\mathcal T_h)<0\}$ and
the Steklov eigenvalues $\{\mu_n< h^{-1/2}\}$\,:
\[ \mu_{n}\sim h^{-1}\sqrt{h+\lambda_{n}(\mathcal T_h)}   ~{\rm ~for~} h^{-1/4}\ll k_n\leq \frac{L}\pi h^{-1/2}+\mathcal O(h^{1/2})  \,.\]
 As a direct consequence of Theorem~\ref{thm:main**}, we obtain   a Weyl law extending   earlier results  \cite{HKR, KPR, KN}.
\begin{theorem}\label{thm:weyl}
Assume that $\Omega$ is simply connected. Let $\epsilon\in[0,\lambda_2^N(\Omega))$. For all $h>0$ and $\lambda\in\R$, we denote by
\[\mathsf N(\mathcal T_h,\lambda):={\rm tr}\Big(\mathbf 1_{(-\infty,\lambda)}\big(\mathcal T_h \big)\Big)\,.\]
Then we have the following asymptotics as $h\to0_+$, 
\[\mathsf N(\mathcal T_h,\epsilon h^2) = \frac{|\Gamma|}{\pi} h^{-1/2}+\mathcal O(h^{-1/4})\,,\]
Furthermore, 
\[\mathsf N(\mathcal T_h,\lambda h) = \frac{|\Gamma|}{\pi} \sqrt{1+\lambda}\,h^{-1/2}+\mathcal O(h^{-1/4})\,,\]
holds for all $\lambda\in(-1,0)$.
\end{theorem}
The asymptotics of $\mathsf N(\mathcal T_h,\lambda h)$ and $\mathsf N(\mathcal T_h,\epsilon h^2)$ hold uniformly with respect to $\lambda\in(-1,0)$ and $\epsilon\in[0,\lambda_2^N(\Omega))$ respectively.
Noting that 
\[\mathsf N(\mathcal T_h,0)=\mathsf N(\Lambda(0),h^{-1/2})\,,\]
we recover the leading order term for the existing results on the DN operator (see  \cite[Eq.~(2.1.4)]{GiI-b})
\begin{equation}\label{eq:N<0-Ro}
\mathsf N(\Lambda(0),h^{-1/2}) =\frac{|\Gamma|}{\pi}h^{-1/2}+ \mathcal O(1)\,.
\end{equation}
The asymptotics in \eqref{eq:N<0-Ro} continues to hold for the generalized DN operator $\Lambda(\epsilon)$  introduced in \eqref{eq:DtN}  if $\epsilon < \lambda_2^D(\Omega)$ is fixed (or in a compact interval of $(-\infty, \lambda_2^D(\Omega)$). Moreover, $\mathsf N(\Lambda(\epsilon),h^{-1/2})=\mathsf N(\mathcal T_h,\epsilon h^2)$, hence we get
\[\mathsf N(\mathcal T_h,\epsilon h^2) =\frac{|\Gamma|}{\pi}h^{-1/2}+ \mathcal O(1)\]
which gives a more accurate estimate of the remainder  than the one appearing in Theorem~\ref{thm:weyl}.

\subsection*{Organization of the paper}
\begin{itemize}
\item[--] In Sec.~\ref{sec:Agm}, we show how we can extract pointwise bounds on the eigenfunctions from the Agmon decay estimates.
\item[--] In Sec.~\ref{sec:PD}, we use a pseudo-differential calculus to prove Theorems~\ref{thm:conj} and \ref{thm:conj-GT}.
\item[--] In Sec.~\ref{sec:1d} we analyze 1D operators that we use later in Sec.~\ref{sec:eff-op} to derive an effective operator for the Robin Laplacian and prove Theorem~\ref{thm:main}.
\end{itemize}

\section{Pointwise bounds via Agmon estimates} \label{sec:Agm}

Using the elliptic and Agmon estimates, we can derive pointwise bounds on the low-energy eigenfunctions of the semi-classical Robin Laplacian operator $\mathcal T_h$.  This was standard in the case of Dirichlet case but because the Robin condition includes the parameter inside the boundary condition, we feel that it is useful to give the details in this new case.

\begin{theorem}\label{thm:dec-agm-pw}
Given $M\in(0,1)$ and $\alpha\in(0,\sqrt{ M})$, there exist positive constants $\varepsilon_0,h_0,C>0$ such that, if $h\in(0,h_0)$ and $u_h$ is a solution of
\[\begin{cases}
-\Delta u_h= w \, u_h &~{\rm in~}\Omega\medskip\\
\displaystyle\frac{\partial u_h}{\partial \nu} =h^{-1/2} u_h&~{\rm on~}\Gamma
\end{cases}\]
with $w<-Mh^{-1}$ and $\|u_h\|_{L^2(\Gamma)}=1$, then 
\begin{equation}\label{eq:dec-pw1}
|u_h(x)|\leq Ch^{-\frac{\mathsf n}4-\frac14}\exp\left(-\frac{\alpha\min\big( d(x,\Gamma),\varepsilon_0 \big)}{h^{1/2}}\right)\quad(x\in\Omega)\,.
\end{equation}
\end{theorem}
\begin{proof}
For all $\varepsilon>0$, we introduce the tubular neighborhood of the boundary,
\begin{equation}\label{eq:Om-ep}
\Omega_{\varepsilon}:=\{x\in\Omega,~d(x,\Gamma)<\varepsilon\}\,.
\end{equation}
Choose $\varepsilon_0>0$ so that the function $x\mapsto d(x,\Gamma)$ is smooth on $\Omega_{2\varepsilon_0}$. We extend this function to a smooth function $\tilde t$ on $\Omega$ as follows
\[ \tilde t(x)=\begin{cases}
d(x,\Gamma)&{\rm if}~x\in\Omega_{\varepsilon_0}\\
2\varepsilon_0&{\rm if~}x\in\Omega\setminus \Omega_{2\varepsilon_0}
\end{cases}~{\rm and}~\varepsilon_0\leq \tilde t(x)\leq 2\varepsilon_0{~\rm if~}x\in \Omega_{2\varepsilon_0}\setminus\Omega_{\varepsilon_0}\,.\]
We introduce the function $v_h(x)=u_h(x)\exp\big(\frac{\alpha \tilde t(x) }{h^{1/2}}\big)$. We select $\alpha\in(0,\sqrt{M})$ and $h_0>0$  so that, for all $h\in(0,h_0)$, \eqref{eq:dec1} holds, which in turn yields
\[\|v_h\|_{H^1(\Omega)}+h^{-1/2}\|v_h\|_{L^2(\Omega)}\leq C_1 h^{-1/2}\|u_h\|_{L^2(\Omega)}\,.\]
The function $v_h$ satisfies the non-homogeneous Neumann problem:
\[ \Delta v_h=f_h{\rm ~in~}\Omega {\rm ~and~}\frac{\partial  v_h}{\partial\nu}=g_h~{\rm on~}\Gamma\,,\]
where
\begin{equation*}
 f_h(x)=\big( \alpha h^{-1/2}\Delta\tilde t-2\alpha^2h^{-1}|\nabla\tilde t|^2-w\big)v_h+2\alpha h^{-1/2}\nabla\tilde t\cdot\nabla v_h  \,,
 \end{equation*}
 and 
 \begin{equation*}g_h(x)=(\alpha+1)h^{-1/2}v_h(x)\,.
 \end{equation*}
By the elliptic estimates for the Neumann non homogeneous problem, we get
\[\|v_h\|_{H^2(\Omega)}\leq C_2\big(\|f_h\|_{L^2(\Omega)}+\|v_h\|_{L^2(\Omega)}+\|g_h\|_{H^{1/2}(\Gamma)} \big)
\leq \tilde C_2 h^{-1}\|u_h\|_{L^2(\Omega)}\,.\] 
In the cases $n=2,3$  and by Sobolev embedding, we deduce an estimate in the H\"older norm.  For the case $n \geq 4$, we pick the smallest integer $k_*>\frac{n}2$ and we iterate the previous estimate so that
\begin{multline*}
\|v_h\|_{H^{k_*}(\Omega)}\leq C_*\big(\|f_h\|_{H^{k_*-2}(\Omega)}+\|v_h\|_{H^{k_*-2}(\Omega)}+\|g_h\|_{H^{k_*-2+1/2}(\Gamma)} \big)\\
\leq \tilde C_*h^{-k_*/2} \|u_h\|_{L^2(\Omega)}\,.
\end{multline*}
We use Sobolev embedding of $H^{k_*}(\Omega)$ in $C(\overline{\Omega})$ and that $k_*\leq \frac{n}2+1$. To finish the proof, we note that   due to our normalization of $u_{h_{/\Gamma}}$, the norm of $u_h$ in $\Omega$ satisfies   $\|u_h\|_{L^2(\Omega)}=\mathcal O(h^{1/4})$ since
\[ -Mh^{-1}\|u_h\|^2_{L^2(\Omega)}\geq -w\|u_h\|_{L^2(\Omega)}^2=\|\nabla u_h\|_{L^2(\Omega)}^2-h^{-1/2}\|u_h\|_{L^2(\Gamma)}^2\,.\]
\end{proof}

\section{Boundary pseudo-differential calculus and decay of eigenfunctions}\label{sec:PD}

\subsection{Decay in the interior}

Here we discuss  (and improve) the weaker result of \cite{HiLu} leading to the conjecture proved by \cite{GT}.  

\begin{theorem}[Hislop-Lutzer \cite{HiLu}]\label{thm:Hi-Lu}\

For any $p\in \mathbb N$, any $K\subset \Omega$ compact, there exists $C_{K,p}>0$ and $h_{K,p} >0$ such that if $(h,u_{h})$ is a solution of \eqref{eq:Robinw=0} 
\begin{equation}\label{eq:Robinw=0}
\begin{cases}
-\Delta u_h=0 &~{\rm in~}\Omega\medskip\\
\displaystyle\frac{\partial u_h}{\partial \nu} =h^{-1/2} u_h&~{\rm on~}\Gamma
\end{cases}\,,
\end{equation}
with $h\in (0,h_{K,p}]$ and $||u_h||_{L^2(\partial\Omega)}=1$,
 then it satisfies
 \begin{equation}
 |u_h (x)|\leq C_{K,p} \, h^{p/2}\,,\, \forall x \in K.
 \end{equation}
 \end{theorem}
 
Note that our Thorem~\ref{thm:conj} extends the result of Theorem~\ref{thm:Hi-Lu} up to the boundary. The idea is to use the properties of the Poisson kernel of the operator $-\Delta-w$ up to the boundary, while in \cite{HiLu}, the properties of the Poisson kernel were used in the interior of the domain.

\subsection{Proof of Theorem~\ref{thm:conj} for $w=0$}\

 The proof of \cite{HiLu} ($w=0$) is based on the classical Green-Representation Formula for $u_h$ (see \cite[Ch.~2, Sec.~2.2.4]{Ev} for the basic theory)
 \begin{equation}
 u_h(x) = \int_{\partial \Omega} u_h (\cdot) P(x,\cdot)  d\sigma\,,
 \end{equation}
 where $P(x,\cdot)$ is the Poisson kernel defined as follows
 \begin{equation}\label{eq:P(x,y)}
P(x,\cdot)=- \partial_\nu G (x,\cdot)
 \end{equation}
where   the distribution $G(x,y)\in \mathcal D'(\Omega\times \Omega)$ is, given $x\in \Omega$,  the solution of the inhomogeneous  Dirichlet problem
 \begin{equation}
 -\Delta_y G (x,\cdot) =\delta_x\,,\,  G(x,y)=0 \mbox{ for  } y \in \partial \Omega\,.
 \end{equation}\Bk
 The properties of $G$ (which is called the Green function) are rather well known in the case of a smooth boundary (see Theorem 2.3 in \cite{HiLu})  but for the proof of the conjecture, we will need a more precise information for the Poisson kernel 
  for $y\in \partial \Omega$ and $x$ close to $\partial \Omega$).   This is done, at least for $w=0$ in \cite{En} (see also \cite{Kr}).
  
 The proof is  based on the connection with the Dirichlet to Neumann operator $\Lambda(w)$. Indeed, $u_{h/{\partial \Omega}}$ is  an eigenfunction of $\Lambda(w)$ 
 associated with the eigenvalue $h^{-\frac 12}$.  We can then write
 \begin{equation}\label{basicf*}
 u_h(x) =  h^{\frac p 2}\,  (P \circ  \Lambda(w)^p ) ({u_h}_{|_{\partial \Omega}})\,.
 \end{equation}
 For $w=0$, \eqref{basicf*} reads as follows,
  \begin{equation}\label{basicf}
 u_h(x) = h^{\frac p 2}\,  \int_{\partial \Omega} u_h (y) \,\cdot\, (\Lambda(0)^p (y,D_y) P(x,y)) d\sigma\,.
 \end{equation}
 For $x\in K$, it is then easy to get the result obtained in \cite{HiLu}, i.e. the  interior decay estimate of Theorem~\ref{thm:Hi-Lu}. As for the estimate of Theorem~\ref{thm:conj} up to the boundary,  we recall the estimate given by M. Englis  in \cite{En}.
 \begin{theorem}
 Let $\Omega$ be a bounded domain in $\mathbb R^{\sf n}$ with smooth boundary. Then the Poisson kernel $P(x,y)$ admits the following decomposition
 \begin{equation}
 P(x,y) = \frac{c_{\sf n} \,d(x)}{|x-y|^{\sf n}} \left[ F(y, |x-y|, \frac{x-y}{|x-y|}) + H(x,y) |x-y|^{\sf n} \log |x-y|\right]\,,
 \end{equation}
 where
 \begin{itemize}
 \item $d\in C^\infty (\bar \Omega)$, $d>0$ on $\Omega$,
 \item $d(x) = d(x,\partial \Omega)$ for $x$ near $\partial \Omega$,
 \item $c_{\sf n}= \frac{\Gamma(\mathsf n/2)}{\pi^{{\sf n}/2}}$
 \item $F \in C^\infty (\partial \Omega\times \bar R^+\times \mathbb S^{\mathsf n-1})\,,\, F(y ,0, \omega)=1$ for $y \in \partial \Omega$, $\omega \in \mathbb S^{\mathsf n-1}$,
 \item $H \in C^\infty(\bar \Omega \times \partial \Omega)$\,.
 \end{itemize}
 \end{theorem}
 This   implies  in particular the weak version mentioned by Krantz \cite{Kr} which reads, for $ \mathsf n\geq 2$, 
 \begin{equation}\label{Kr1}
 |\partial_y^\alpha P(x,y)| \leq C_\alpha\, d(x) |x-y|^{-\mathsf n-|\alpha|}\,,\, \forall y \in \partial \Omega, x\in \Omega\,.
  \end{equation}
  This last estimate directly implies
  \begin{equation}\label{Kr2}
 |\partial_y^\alpha P(x,y)| \leq C_\alpha \, d(x)^{1-\mathsf n-|\alpha|}\,,\, \forall y \in \partial \Omega, x\in \Omega\,.
  \end{equation}
   Coming back to \eqref{basicf}, we can write for $p$ even (if we do not want to use the complete Boutet de Monvel  calculus)
    \begin{multline}\label{basicf1}
 u_h(x)=  h^{\frac p 2}\,  \int_{\partial \Omega} u_h (y) \,\cdot\, (\Lambda(0)^p\cdot (-\Delta_y)^{-p/2}) \,  ((-\Delta_y)^{p/2})P(x,y)d\sigma\\
=  h^{\frac p 2}\,  \int_{\partial \Omega} (\Lambda(0)^p\cdot (-\Delta_y)^{-p/2}) u_h (y) \,\cdot\,  ((-\Delta_y)^{p/2})P(x,y) d\sigma\,.
 \end{multline}
 We now observe that $ (\Lambda(0)^p\cdot (-\Delta_y)^{-p/2})$ is a boundary pseudodifferential operator of degree $0$ (with constant principal symbol) and using \eqref{Kr2} we obtain, for any $p\geq 1$, 
 \begin{equation}\label{HLimproved}
 | u_h(x)| \leq  C_p \,  h^{\frac p 2} d(x)^{1-\mathsf n- p}\,.
  \end{equation}
  This proves Theorem~\ref{thm:conj} for $w=0$.

\subsection{Proof of Theorem~\ref{thm:conj} for $w\in [-\pi^2,\lambda_1^D(\Omega))$}\

Now assume that $-\pi^2\leq w<\lambda_1^D(\Omega)$ and $w\not=0$.  The proof is similar to the case $w=0$ but we should replace the Green function $G$ by $G_w$ and the ND operator $\Lambda(0)$ by $\Lambda(w)$. There is no problem of definition if $w$ is not an eigenvalue of the Dirichlet Laplacian.  To avoid to analyze if the proof written for $w=0$ goes on, we can use a weaker theorem which holds for general  potential operators (or Poisson like operators). See  \cite[Thm.~8, p.~18]{En} with $n\geq 2$, $d=p$ and observe that  $|x-\zeta| ^{-1}\leq d(x)^{-1}$.
The aforementioned  result of \cite{En}  reads as follows:

  \begin{theorem}[Englis \cite{En}]\label{thm:Englis2}
  If $K$ is a potential operator in $ \mathcal K^d(\overline{\Omega})$, where $\Omega \subset \mathbb R^{\mathsf n}$ is a bounded domain with smooth boundary (or $\Omega =\mathbb R_+^{\mathsf n}$), then the Schwartz kernel $k_K$ satisfies, if $d\in \mathbb Z, d>1-{\mathsf n}$
  \begin{equation}
  k_K (x,y)=  |x-y|^{1-{\mathsf n}-d}  F(y, |x-y|, \frac{x-y}{|x-y|}) + H(x,y)  \log |x-y|\,,
  \end{equation}
  where $F$ and $H$ have the same property as in the previous theorem.
  \end{theorem}
  
In our application, we use that the Poisson operator (associated with $(-\Delta -w)$) is a potential operator $P(w)$  if $w$ is not an eigenvalue of the Dirichlet Laplacian. We also use 
   the property that the Dirichlet to Neumann operator $\Lambda(w)$ is a boundary pseudo-differential operator of degree $1$ with elliptic principal symbol.

We apply Theorem~\ref{thm:Englis2} to $K=  (P \circ  \Lambda(w)^p)$ and  use  \eqref{basicf}.
Everything depends continuously of $w$ in the interval $I:=[-\pi^2,\lambda^D_1(\Omega))$ and the control is uniform in any compact interval in $I$.
This is clear for the computation (symbolic calculus)  of an approximate Poisson operator $ P^{app}(w) $  modulo regularizing operators $R^{reg} (w)$ and $r^{reg}(w)$  without additional  assumptions.
One gets
$$
(-\Delta - w ) P^{app}(w)  = R^{reg} (w) \,,\, \gamma \circ P^{app} (w) = I + r ^{reg}(w)\,.
$$
For eliminating the remainder, we use the resolvent and this is there  that the assumption that $w$ is not in the spectrum of  the Dirichlet Laplacian is used. 
More precisely, we first eliminate $r(w)$ by using simply an extension operator $\epsilon$ from $C^\infty (\partial \Omega)$ into $C^\infty (\overline {\Omega})$.
We note that $\epsilon \circ r^{reg} (w)$ is regularizing.\\
Then, we compute
$$
(-\Delta -w) ( P^{app}(w) - \, \epsilon\, \circ\, r^{reg}(w))= R^{reg} (w) - (-\Delta -w) \epsilon \circ r^{reg}(w)) := \hat R^{reg} (w)\,.
$$
Finally, we get for the Poisson kernel
$$
P (w) = P^{app}(w) -\epsilon\, \circ \, r^{reg} (w) - (-\Delta-w)^{-1} \hat R^{reg} (w)\,.
$$
At this stage we  get \eqref{HLimproved} from \eqref{basicf*} in the case where $w\not=0$ is fixed in the interval $[-\pi^2,\lambda_1^D(\Omega))$,  the estimate being uniform in $w$ for any compact subinterval.  We have the same result for any compact interval in the resolvent set of the Dirichlet Laplacian in $\Omega$. The choice of $-\pi^2$ is only  motivated by the next step.
 
\subsection{Proof of Theorem~\ref{thm:conj} for $w<-\pi^2$} 
 
The problem here is that we loose  in the previous approach the control of the  uniformity with respect to $w$ in the estimates of the Poisson kernel $P(w)$. Actually, since $h^{-1}w\in\sigma(\mathcal T_h)$,  $w=w(h)$ may approach $-\infty$ in the semi-classical limit, by Theorem~\ref{thm:main}. 

Pick the unique integer $ k\geq 1$  such that 
\[
k\pi\leq \sqrt{-w}<(k+1)\pi
\]
and set
\[
a=\frac{k\pi}{\sqrt{-w}}\,.
\]
Then, 
\begin{equation}\label{eq:w-a-n}
w+\frac{k^2\pi^2}{a^2}=0,~k\in\mathbb N,~\frac12\leq a\leq 1\,.
\end{equation}
We introduce the  a weighted Laplace operator $-\Delta_{\hat\Omega,a}$ in the cylinder  $\hat\Omega:=\Omega\times \mathbb T^1$, where $\mathbb T^1=[0,1)$ is the 1d torus. That is
\begin{equation}\label{eq:D-strip}
\Delta_{\hat\Omega,a}=\sum_{i=1}^2\frac{\partial^2}{\partial x_i^2}+ \frac1{a^2} \frac{\partial^2}{\partial\theta^2}
\end{equation}
where $(x_1,x_2)$ denote the coordinates in $\Omega$    and $\theta$ denotes the coordinate in $\mathbb T^1=[0,1)$; these  coordinates represent a point $\hat x=(x,\theta)$ of $\hat\Omega$.
We introduce the following function
\begin{equation}\label{eq:D-strip-ef}
v_h(\hat x)=e^{ik\pi\theta s}u_h(x)~\big(\hat x=(x,\theta)\big)
\end{equation}
which satisfies
\begin{equation}\label{eq:vh-strip}
-\Delta_{\hat\Omega,a} v_h=0~{\rm on~}\hat\Omega,~\nu_{\hat\Gamma}\cdot\nabla_{\hat x} v_h=h^{-1/2} v_h~{\rm on~}\hat\Gamma\,.
\end{equation}
Here   $\hat\Gamma=(\partial\Omega)\times \mathbb T^1$ is the  boundary of $\hat\Omega$, and $\nu_{\hat\Gamma}$ its unit outward normal vector; we have $\nu_{\hat\Gamma}=(\nu^1,\nu^2,0)$ where $\nu=(\nu^1,\nu^2)$ is the outward unit normal vector  of $\Gamma=\partial\Omega$. . 

We can introduce the DN operator of $\hat\Omega$, $\Lambda_{\hat\Omega,a}(0)$, defined on $H^{1/2}(\hat\Gamma)$ as in \eqref{eq:DtN} (with $\hat\Omega,\hat\Gamma$ replacing $\Omega,\Gamma$ and $-\Delta_{\hat\Omega,a}$ replacing $-\Delta$). Consequently, the function $v_h$ in \eqref{eq:vh-strip} is an eigenfunction of $\Lambda_{\hat\Omega,a}(0)$ with eigenvalue $h^{-1/2}$.
We will use  the Poisson kernel $P_{\hat\Omega,a}$ corresponding to $-\Delta_{\hat\Omega,a}$. 
Using Theorem~\ref{thm:Englis2} for the domain $\hat\Omega$ and the operator $-\Delta_{\hat\Omega,a}$,  we get the following  Poisson kernel estimates (as in \eqref{Kr1}--\eqref{Kr2}) 
\begin{equation}\label{eq:Kr-torus2} 
 |\partial^\alpha_{\hat x} P_{\hat\Omega,a}(\hat x, \hat y)|\leq C_\alpha d(\hat x ,\hat\Gamma)^{-\mathsf n-|\alpha|}\quad(\hat x\in\hat\Omega,\hat y\in\hat\Gamma)
\end{equation}
where the constant $C_\alpha$ can be chosen independently of $a\in[\frac12,1]$.

With \eqref{eq:vh-strip} in hand, we write, for $\hat x=(x,\theta)\in\hat\Omega$,
\begin{align*}
 v_h(\hat x)&=\int_{\hat\Gamma}v_h(\hat y)P_{\hat\Omega,a}(\hat x,\hat y)d\sigma_{\hat\Gamma}(\hat y)\\
 &=h^{p/2}\int_{\hat\Gamma} \Lambda_{\hat\Omega,a}(0)^p v_h(\hat y)P_{\hat\Omega,a}(\hat x,\hat y)d\sigma_{\hat\Gamma}(\hat y)\,.
\end{align*}
 Using the Poisson kernel estimate in \eqref{eq:Kr-torus2} and the pseudodifferential calculus as in \eqref{basicf1}, we get, for any positive even integer $p$, any $a\in [\frac 12,1]$, the existence of $C_p$ and $h_p>0$ such that, for $h\in (0,h_p]$,
\[|v_h(\hat x)|\leq C_p \,  h^{p/2} d(\hat x,\hat\Gamma)^{-\mathsf n-p}\,. \]
Since $|v_h(\hat x)|=|u_h(x)|$ by  \eqref{eq:D-strip-ef} and $d(\hat x,\hat\Gamma)=d(x,\partial\Omega)=d(x)$ for $\hat x=(x,\theta)\in\hat\Omega$, this implies 
  \begin{equation}\label{eq:D-strip-gr}
 |u_h(x)|\leq C_p \,  h^{p/2} d(x)^{-\mathsf n-p}\,,\,\forall x\in \Omega\,,
 \end{equation}
 as stated in Theorem \ref{thm:conj} for $w\leq - \pi^2$.

\subsection{Analytic case}\label{rem:exp-dec}
 We now consider the  case when $\partial\Omega$ is analytic and handle the case where $w<\lambda_1^D(\Omega)$. 
\subsubsection{Using analytic pseudodifferential calculus}
At a  heuristic level, one could  hope from the Boutet de Monvel  analytic pseudodifferential calculus \cite{BM1} that we will get an estimate in the form
 \begin{equation}\label{HLimproveda}
 | u_h(x)| \leq  C ^{p+1} \, p!  \, h^{\frac p 2} d(x)^{1-n- p}\,.
  \end{equation}
   A first step could be the following (to our knowledge unproved)  result:
  If  $A$ is an analytic pseudo-differential operator on  $\partial \Omega$ (or more generally a compact analytic manifold) of degree  $1$ 
   and $u$ is an analytic function on $\partial \Omega$, then  $A^p u$ satisfies
   $$
   |(A^p u) (y) | \leq C^{p+1} p!\,.
   $$
   This kind of estimate (with additional control with respect to the distance of $x$ to $\partial \Omega$)  should be applied to the distribution kernel of the Poisson operator of $-\Delta-w$. 
   
  Assuming that the estimate  \eqref{HLimproveda} is true we can try to optimize over $p$.  Using Stirling Formula, we get the simpler 
   \begin{equation}\label{HLimprovedb}
 | u_h(x)| \leq  C ^{p+1} \, p^{p+1} \, h^{\frac p 2} d(x)^{1-n- p}\,.
  \end{equation}
  Optimizing over $p$ will give an estimate of the form \eqref{HLimprovedc}.

  It seems difficult by this approach to have the optimal result of Galkowski-Toth \cite{GT}, i.e. to have a control   of the constant $C_1$ appearing in \eqref{HLimprovedc}.
  
We also refer the reader to an interesting discussion at the end of \cite{En} (Subsection 7.4) and to \cite{PST}.

  \subsubsection{Using Galkowski-Toth.}

 In this section, we prove Theorem~\ref{thm:conj-GT}.  \Bk  To keep tracking the uniformity with respect to $w$  of the estimates, we introduce a fixed positive constant $0<\zeta <\lambda_1^D(\Omega)$ and assume that $w$ varies as follows, $-\infty<w\leq \epsilon$.

 We recall Theorem~1 of Galkowski-Toth \cite{GT}:

\begin{theorem}[\cite{GT}]\label{thm:GT}
For all $\delta>0$ and $\alpha \in \mathbb N^n$, there exist  $\varepsilon >0$,  $C$  and $h_0$  such that, for $h\in (0,h_0]$,  any solution $u_h$ of
\begin{equation}\label{eq:GT0}
\begin{cases}
-\Delta u_h=0 &~{\rm in~}\Omega\medskip\\
\displaystyle\frac{\partial u_h}{\partial \nu} =h^{-1/2} u_h&~{\rm on~}\Gamma\medskip\\
\|u_h\|_{L^2(\partial\Omega)}=1
\end{cases}\,,
\end{equation}
satisfies the following estimate in $\{d(x,\Gamma)<\varepsilon \}$,
\begin{equation} \label{eq:GT0a}
|\partial_x^\alpha u_h(x)|\leq C \, h^{-\frac{n}4+\frac18-\frac{|\alpha|}2}\exp\left(-\frac{d(x,\Gamma)\big(1+(C_{\Omega}-\delta)d(x,\Gamma) \big) }{h^{1/2}}\right)\,. 
\end{equation}
 Here $C_{\Omega}=-\frac32+\inf_{(x',\xi')\in S^*\Gamma}Q(x',\xi')$, $Q$ is the symbol of the second fundamental  form of the boundary $\Gamma$.
\end{theorem}
It results from Theorem~\ref{thm:GT} the following weaker estimate. There exist constants $\varepsilon, C,\hat C$ such that, for $d(x,\Gamma)<\varepsilon$, we have 
\begin{equation} \label{eq:GT0a*}
|\partial_x^\alpha u_h(x)|\leq C \, h^{-\frac{n}4+\frac18-\frac{|\alpha|}2}\exp\left(-\frac{d(x,\Gamma)\big(1- \hat C  d(x,\Gamma) \big) }{h^{1/2}}\right)\,. 
\end{equation}
Looking at the  proof, Theorem~\ref{thm:GT} can be generalized in two different ways:
\begin{itemize}
\item When replacing $-\Delta$ by $-\Delta-w$,  the constants in the estimates can be controlled uniformly with respect to $w$ in any compact interval of $(-\infty,\lambda_1^D(\Omega))$.
\item When  replacing $-\Delta$ by ${\rm div}(\mathbf c\nabla)$ with $\mathbf c\in\R^n$ a constant vector with positive components,  the constants in the estimates can also be controlled uniformly with respect to $|\mathbf c|$ when it varies in a compact interval in $(0,+\infty)$.
\end{itemize} 
In the two aforementioned  situations,  \eqref{eq:GT0a} continues to hold, which also yields that,  for all $\eta>0$, there exist positive constants  $\varepsilon ,C, h_0$  such that, for $h\in (0,h_0]$,  any solution $u_h$ of \eqref{eq:GT0} 
satisfies the following estimate in $\{d(x,\Gamma)<\varepsilon \}$,
\begin{equation}\label{eq:GT-gen}
  |u_h(x)|\leq C \, h^{-\frac{n}4+\frac18}\exp\left(- (1-\eta) \frac{d(x,\Gamma)}{h^{1/2}}\right)\,.
\end{equation}
Note that we just keep \eqref{eq:GT0a*} which is the weaker version of \eqref{eq:GT0a} for simplification. In the procedure of addition of one variable described  below, we can not keep the additional information related to the curvature of  $\Gamma$, but we can always write the following estimate (which also leads to \eqref{eq:GT-gen}):

\textit{There exist positive constants  $C,\hat C,  h_0$  such that, for all $h\in (0,h_0]$,   
\begin{equation}\label{eq:GT-gen*}
  |u_h(x)| \leq C \, h^{-\frac{n}4+\frac18} \exp\left(- \frac{ d_{\hat C} (x)   }{h^{1/2}}\right) \,,
\end{equation}
with 
\begin{multline*}
d_{\hat C} (x) = (d(x,\Gamma) -\hat C d(x,\Gamma)^2) 1_{ \{y, \, d(y,\Gamma) < \frac{1}{2\hat C}\}} (x)\\
 + \frac{1}{2 \hat C}\left(1-  1_{ \{y,\,  d(y,\Gamma) < \frac{1}{2\hat C}\}} (x)\right) \,.
\end{multline*} }

We proceed with the proof of Theorem~\ref{thm:conj-GT}. We start with the case $w<-\pi^2$ and  apply the Galkowski-Toth estimate \eqref{eq:GT-gen} for  the solution $v_h$ of \eqref{eq:vh-strip}. We get
\[|v_h(\hat x)|\leq C_{\alpha}h^{-\frac12\big(\frac{n}{2}+\frac34\big)}\exp\left(-(1-\eta)\frac{d(\hat x,\hat\Gamma)}{h^{1/2}} \right)\,, \]
in a tubular neighborhood $ \hat\Omega_\varepsilon=\{x\in\hat\Omega,~{\rm dist}(\hat x,\hat\Gamma)<\varepsilon\}$. Note that the second fundamental  form of $\hat\Omega$ vanishes so the estimate does not display the effect of the curvature of $\Omega$ as in \eqref{eq:GT0}. 

Remarking that $|v_h(\hat x)|=|u_h(x)|$  and  $d(\hat x,\hat\Gamma)=d(x,\Gamma)=d(x)$ for $\hat x=(x,\theta)\in\hat\Omega$, we get 
\[| u_h( x)|\leq Ch^{-\frac12\big(\frac{n}{2}+\frac34\big)}\exp\left(-(1-\eta)\frac{d( x,\Gamma)}{h^{1/2}} \right)\,, \]
in $\Omega_\varepsilon$. To get the interior estimate
\[|\hat v_h(\hat x)|\leq \hat C \exp(-\hat c \, h^{-1/2})~{\rm in~}\Omega\setminus \hat\Omega_\varepsilon\,, \]
we use the maximum principle, for the operator $-\Delta_{\hat\Omega,a}$  and  the solution $\hat v_h$,   in  $\Omega\setminus \hat\Omega_\varepsilon$
(see \cite[Lem.~3.2.9]{PST} for the details of the argument). This finishes the proof   of \eqref{eq:GT-gen}   for $w<-\pi^2$.

We move now to the case where $-\pi^2\leq w\leq \zeta$. We use the estimate \eqref{eq:GT-gen} for the solution of $-\Delta u_h=wu_h$ and get 
\[ |u_h(x)|\leq C  h^{-\frac{n}4+\frac18}\exp\left(-(1-\eta)\frac{d(x,\Gamma) }{h^{1/2}}\right)~{\rm in~}\Omega_\varepsilon\,.\]
If moreover $w\leq 0$,  we use the maximum principle, as in \cite{GT, PST} to get the interior estimates. Notice that we use the maximum principle for the operator  $-\Delta-w$ with $w\leq 0$ so that the arguments of \cite{GT,PST} hold\,\footnote{see for example Stampacchia \cite[Thm.~3.8]{St} for the maximum principle for $-\Delta-w$ when $w\leq0$}. 

If $0 < w \leq \zeta< \lambda_1^D(\Omega)$, we apply the maximum principle to the function $f_h$ defined by $u_h=f_h  u^D$, where $u^D$ is the  normalized positive ground state of the Dirichlet Laplacian on $\Omega$.  The function $f_h$ satisfies   
\[-\frac1{(u^D)^2}{\rm div}\Big((u^D)^2\nabla f_h\Big)+c f_h=0~{\rm with~}c:=\lambda_1^D(\Omega)-w> 0\,.\]

\section{One dimensional Robin Laplacians}\label{sec:1d}
We revisit one dimensional model operators appearing in \cite{HK-tams}.

\subsection{On the half line}
We start with the self-adjoint operator in $L^2(\R_+)$ defined by
\begin{equation}\label{defH00}
\mathcal H_{0}=-\partial^2_{\tau}
\end{equation}
on the domain
\begin{equation}
\Dom(\mathcal H_{0})=\{u\in H^2(\R_+)~:\,u'(0)=- u(0)\}\,.
\end{equation}
The quadratic form associated with  this operator is 
\[
V_0 \ni u\mapsto \int_0^{+\infty} |u'(\tau)|^2\,d\tau\,  - |u (0)|^2\,,
\]
with form domain $V_0 = H^1(0,+\infty)\,$.

The spectrum of this operator is (see \cite{HK-tams})
$$
\sigma (\mathcal H_0) =\{-1\}\cup[0,+ \infty)\,
$$
and
the eigenvalue $-1$ has multiplicity one with the corresponding  $L^2$-normalized positive eigenfunction
\begin{equation}\label{eq:u0}
 u_1(\tau)=\sqrt{2}\,\exp\left(-\tau\right).
\end{equation}

\subsection{On an interval}
Let us consider $T\geq 1$ and the self-adjoint operator acting on $ L^2(0,T)$ and defined by
\begin{equation}\label{eq:H0}
\mathcal H^{T,D}_{0}=-\partial^2_{\tau}\,,
\end{equation}
with domain,
\begin{equation}\label{eq:DomH0}
\Dom(\mathcal H^{T,D}_{0})=\{u\in H^2(0,T)~:~u'(0)=-u(0)\quad{\rm and}\quad u(T)=0\}\,.
\end{equation}
This operator is associated with the quadratic form
\[
V^{T,D}_0\ni u\mapsto \int_0^{T} |u'(\tau)|^2\,d\tau\,  - |u (0)|^2\,,
\]
with $V^{T,D}_0 =\{v\in H^1(0,T)\,|\, v(T)=0\}$.\\

The spectrum of the operator $\mathcal H^{T,D}_{0}$ is purely discrete. We denote by
$\left(\lambda_n^D(T)\right)_{n\geq 1}$ the sequence of min-max eigenvalues and  by $(u^{T,D}_{n})_{n\geq 1}$ some associated  $L^2(0,T)$ Hilbert basis of eigenfunctions.  We can localize the spectrum in the large $T$ limit \cite[Lem.~4.1\, \mbox{ and } \,Rem.~4.3]{HK-tams} and \cite[Lem.~A.2]{KPR}.

\begin{lemma}\label{lem:1DL-D}
 As $T \to +\infty$, it holds
\begin{equation}\label{eq:lwh}
\lambda^{T,D}_1(T)= - 1 + 4 \big(1+o(1)\big) \exp\big( - 2 T\big)\,,
\end{equation}
and the eigenfunction $u_1^{T,D}$ satisfies
\begin{equation}\label{eq:1DL-ef}
\big\|e^{\tau}\big(u_1^{T,D}-u_1\big)\big\|_{W^{1,\infty}(0,T)}=\mathcal O\big( T\big)\,,
\end{equation}
where $u_1$ is the eigenfunction in \eqref{eq:u0}.

Furthermore, for all $T>1$ and  $n\geq 2$, we have
\[\left(\frac{(2n-3)\pi}{2T}\right)^2<\lambda^{D}_n(T)<\left(\frac{(n-1)\pi}{T}\right)^2\,.\]
\end{lemma}

Also we consider the Neumann realization at the boundary $t=T$,
\begin{equation}\label{eq:H1}
\mathcal H^{T,N}_{0}=-\partial^2_{\tau}\,,
\end{equation}
with domain,
\begin{equation}\label{eq:DomH1}
\Dom(\mathcal H^{T,N}_{0})=\{u\in H^2(0,T)~:~u'(0)=-u(0)\quad{\rm and}\quad u'(T)=0\}\,.
\end{equation}
The spectrum of the operator $\mathcal H^{\{T\}}_{1}$ is purely discrete, consisting of the sequence of min-max eigenvalues $\left(\lambda_n\left(\mathcal H^{T,N}_{0}\right)\right)_{n\geq 1}$. We denote by $(u^{T,N}_{n})_{n\geq 1}$ the corresponding Hilbert basis of eigenfunctions.
We can localize the spectrum in the large $T$ limit.

\begin{lemma}\label{lem:1DL-N}
As $T \to +\infty$, it holds
\begin{equation}\label{eq:lwh-N}
\lambda^N_1(T)= - 1 + 4 \big(1+o(1)\big) \exp\big( - 2 T\big)
\end{equation}
and
\begin{equation}\label{eq:1DL-ef} 
\big\| e^{\tau}\big(u_1^{T,N}-u_1\big)\big\|_{W^{1,\infty}(0,T)}=\mathcal O\big( T\big)\,,
\end{equation}
where $u_1$ is the eigenfunction in \eqref{eq:u0}.

Furthermore, for all $T>1$ and  $n\geq 2$, we have
\[\left(\frac{(2n-3)\pi}{2T}\right)^2<\lambda^N_n(T)<\left(\frac{(n-1)\pi}{T}\right)^2\,.\]
\end{lemma}
\begin{proof}
The proof is similar to that of Lemma~\ref{lem:1DL-D} but we give the main points for the convenience of the reader. Let $\lambda\leq 0$ be a non-positive  eigenvalue of the
operator $\mathcal H^{T,N}_0$ with an eigenfunction $f$. Solving the ODE $f''=\lambda f$ with the boundary conditions $f'(0)=-f(0)$ and $f'(T)=0$ yields that the only possible value of $\lambda$ is 
\[\lambda=-1+ 4 \big(1+o(1)\big) \exp\big( - 2 T\big),\] which corresponds to the first eigenvalue (see \cite{HK-tams}). The corresponding normalized eigenfunction is
\[u_1^{T,N}(\tau)= A_T \Big(e^{-\tau}+e^{-2\alpha\Bk  T}e^{\tau}\Big) \]
with $A_T=\sqrt{2}+\mathcal O(Te^{-T})$ and $\alpha\Bk=1-2(1+o(1))e^{-2T}$, so that 
\[e^{-2\alpha T}e^{\tau}=\mathcal O (e^{\tau-2 T})=o(e^{-\tau})\,.\]
We then have the following uniform estimate,
\[ |u_1^{T,N}(\tau)-A_T e^{-\tau}|= \mathcal O\big(e^{\tau-2T}\big) \]
which also yields
\[ \big|e^\tau\big( u_1^{T,N}(\tau)-\sqrt{2}e^{-\tau}\big)\big|=\mathcal O ( T)\,.\]
Although not needed here, note that we have the much more accurate approximation
$$
|A_T^{-1} u_1^{T,N}(\tau)- (e^{-\tau}+ e^{\tau-2T}) |= \mathcal O (T e^{-3 T})\,.
$$
Now we study the positive eigenvalues.  Let $ \ell>0$ and $\lambda=\ell^2$ be a non-negative  eigenvalue of the
operator $\mathcal H_{0}^{T,N}$ with an eigenfunction $u$, which will have the form
\[
u(\tau)=A\cos(\ell\tau)+B\sin(\ell\tau)\,,
\]
for some constants $A,B\in\R$  that depend on $T$, with $A=-B\ell$,  $\sin(\ell T)\not=0$,    and
$\cot(\ell T)=-\ell$,
 to respect the  boundary conditions.
The positive fixed points of the  $\pi/T$-periodic function $x\mapsto-\cot(xT)$ must belong to the intervals $I_k:=(\frac{\pi}{2T},\frac{\pi}{T})+\frac{k\pi}{T}$, $k=0, 1,\cdots$. In each interval $I_k$, there exists a unique fixed point $\ell_k$ because the function $g(x)=\cot(xT)+x$ satisfies $g'(x)=-T(1+\cot^2(xT))+1<0$ for $T>1$. 
 For each $k=0,1,2,\cdots$, the fixed point $\ell_k\in I_k$ is equal to  $\sqrt{\lambda_{k+2}^{N}(T)}$.
\end{proof}

\subsection{On a weighted space}

Now we  consider operators with weight terms, which can be viewed as perturbations of the operators studied previously on the interval $(0,T)$ with Dirichlet or Neumann realizations at the endpoint $t=T$.

In the sequel, $\rho\in(\frac13,\frac12)$ and $M>0$ are fixed constants, and 
\[ T=T_h:=h^{\rho-\frac12}\,.\] 
We pick    $h_0=h_0(\rho,M)>0$ such that, for all $h\in(0,h_0] $ and $\beta\in[-M,M]$,
we have $\frac12<1-h^{1/2}\beta \tau<1$   for all $\tau\in(0,T)$.

Consider the differential operator
\[
\begin{aligned}
\mathcal H_{h,\beta}&=-(1-h^{1/2}\beta\tau)^{-1}\frac{d}{d\tau}\Big((1-h^{1/2}\beta\tau)\frac{d}{d\tau}\Big)\\
&=-\frac{d^2}{d\tau^2}+\beta h^{1/2}(1-h^{1/2}\beta\tau)^{-1}\frac{d}{d\tau}\,.\end{aligned}\]
We work in the Hilbert space 
\begin{equation}\label{eq:sp-X}
X_{h,\beta}=L^2\big((0,T_h);(1-h^{1/2}\beta\tau)d\tau\big)\,,
\end{equation} 
with inner product and norm defined by
\begin{equation}\label{eq:ip-n-X}
\langle u,v\rangle_{h,\beta}=\int_0^T u(\tau)\overline{v(\tau)}\,(1-h^{1/2}\beta\tau)d\tau,\quad \|u\|_{h,\beta}=\langle u,u\rangle_{h,\beta}^{1/2}\,.
\end{equation} 
Consider the   two self-adjoint  realizations  of $\mathcal H_{h,\beta}$ in $X_{h,\beta}$, $\mathcal H_{h,\beta}^N$ and $\mathcal H_{h,\beta}^{D} $, with domains
\begin{equation}\label{eq:1Dw-evD}
\begin{aligned}
&\mathfrak D^N_{h}=\{u\in H^2(0,T),~u'(0)=-u(0)\,\&\,u'(T)=0\}\\
\mbox{ and } &\\
& \mathfrak D^D_{h}=\{u\in H^2(0,T),~u'(0)=-u(0)\,\&\,u(T)=0\}\,.
\end{aligned}
\end{equation}
We denote the sequences of min-max eigenvalues by $\big(\lambda^N_{n,h}(\beta) \big)_{n\geq 1}$ and $\big(\lambda^D_{n,h}(\beta) \big)_{n\geq 1}$ respectively. 

By the min-max principle, we can localize the foregoing eigenvalues as follows
\begin{equation}\label{eq:2nd-ev-1Dw}
 \big| \lambda^{\#}_{n,h}(\beta)  -\lambda^{T,\#}_{n,h} \big|\leq \big(1+\lambda^{T,\#}_{n,h} \big) h^\rho 
 \end{equation}
 uniformly with respect to $\beta\in[-M,M]$ and $h\in(0,h_0] $. Here $\#\in\{N,D\}$ and $\lambda^{T,\#}_{n,h}$ are the eigenvalues of the operators introduced   in \eqref{eq:H0} and \eqref{eq:H1} with $T=h^{\rho-\frac12}\gg1$. We deduce then that there exists $h_0>0$ such that for $h\in (0,h_0]$
\begin{equation}\label{eq:1Dev2-big}
 \lambda_{2,h}^{\#}(\beta)\geq \frac{\pi^2}{4}h^{1-2\rho}-\mathcal O(h^\rho)\geq \frac{\pi^2}{8}h^{1-2\rho}>0\,,
\end{equation}
 since $\rho\in(\frac13,\frac12)$.
The first eigenvalue $\lambda_{1.h}^{\#}$ was analyzed in \cite[Prop.~4.5]{HK-tams}  and  \cite[Lem.~2.5]{KS} for the Dirichlet case ($\#=D$). The same analysis applies for the Neumann case ($\#=N$). 
 we have
\begin{equation}\label{eq:1Dev1}
\Big|\lambda_{1.h}^{\#}(\beta) -\big(-1-\beta h^{1/2}-\frac{\beta^2}{2}h\big)\Big|\leq C(|\beta|^5+1)h^{\frac32}\,,
\end{equation}
uniformly with respect to $\beta\in[-M,M]$ and $h\in(0,h_0] $.

For the convenience of the reader, we present the outline of the proof of \eqref{eq:1Dev1}. The idea is to look for a formal eigenpair of the form
\[u^{\rm app}_{h,\beta}=v_0+h^{1/2}v_1+hv_2~{\rm and~}\mu^{\rm app}_{h,\beta}=\mu_0+\mu_1h^{1/2}+\mu_2h\,.\]
We expand $\big(\mathcal H_{h,\beta}-\mu^{\rm app}_{h,\beta}\big)u^{\rm app}_{h,\beta}(\tau) $ as $L_0+h^{1/2}L_1+hL_2+h^{3/2}r_\beta(\tau)$ with
\begin{align*}
&L_0=\Big(-\frac{d^2}{d\tau^2}-\mu_0\Big)v_0,~L_1=\Big(-\frac{d^2}{d\tau^2}-\mu_0\Big)v_1+\Big(\beta\frac{d}{d\tau}-\mu_1\Big)v_0,\\
&L_2=\Big(-\frac{d^2}{d\tau^2}-\mu_0\Big)v_2+\Big(\beta\frac{d}{d\tau}-\mu_1\Big)v_1+\Big(\beta^2\frac{d}{d\tau}-\mu_2\Big)v_0\\
&|r_\beta(\tau)|\leq C(|\beta|^3+1)(\tau^2+1)\sum\limits_{i=1}^2|v_i(\tau)|
\end{align*}  We choose the pairs $(v_i,\mu_i)$ so that the coefficients $L_0,L_1,L_2$ vanish \cite[Lem.~2.5]{KS}. Eventually we get  the approximate eigenfunction
 $$\mu^{\rm app}_{h,\beta}:=-1-\beta h^{1/2}-\frac{\beta^2}{2}h\,,$$
 and the following quasi-mode 
\begin{equation}\label{eq:1trial-state}
u^{\rm app}_{h,\beta}(\tau) := \Big(1+\beta^2h\Big(\frac{\tau^2}4-\frac18\Big) \Big)u_1(\tau)\,,
\end{equation}
where $u_1$ the eigenfunction in \eqref{eq:u0}.
 The following estimate holds, for all $\tau\in(0,T)$,
\begin{equation}\label{eq:1trial-state*} 
\Big|\Big(\mathcal H_{h,\beta}-\mu^{\rm app}_{h,\beta}\Big)u^{\rm app}_{h,\beta}(\tau)\Big|\leq C h^{\frac32}(|\beta|^5+1)(\tau^2+1)^2 |u_1(\tau)|
\end{equation}
uniformly with respect to $\beta\in[-M,M]$  and $\tau\in(0,T)$. 

We introduce the following quasi-mode (it belongs to $\mathfrak D(\mathcal H_{h,\beta}^\#)$)
\begin{equation}\label{eq:vh-gs=app}
v_h(t)=c_h\, \chi(T^{-1}\tau ) \, u^{\rm app}_{h,\beta}(\tau)\,,
\end{equation}
where $\chi\in C_c^\infty(\R)$ satisfies $0\leq \chi\leq 1$, ${\rm supp}\,\chi\subset (-1,1)$, $\chi_{/_{[-\frac12,\frac12]}}=1$,  and where $ c_h$ is selected so that $\|v_h\|_{h,\beta}=1$. By the exponential decay of $u_1$ (see \eqref{eq:u0}), the constant $c_h$ and  the quasi-mode $v_h$  satisfy
\[ c_h=1+\mathcal O(h^{1/2})\]
and
\begin{equation}\label{eq:1trial-state**} 
\Big|\Big(\mathcal H_{h,\beta}-\mu^{\rm app}_{h,\beta}\Big)v_h(\tau)\Big|\leq \tilde C h^{\frac32}(|\beta|^5+1)(\tau^2+1)^2 |u_1(\tau)|\,.
\end{equation}
The spectral theorem and \eqref{eq:1Dev2-big}  yield the estimate in \eqref{eq:1Dev1}.\\

We will need the following lemma on the `energy' of functions orthogonal to the quasi-mode $v_h$ in the space $X_{h,\beta}$ introduced in \eqref{eq:sp-X}.  

\begin{lemma}\label{lem:gs=app}
Then there exist positive constants $m,h_0$ such that, if $h\in (0,h_0]$ and $g_h\in H^1(0,T)$ is orthogonal to $ v_h$ in $X_{h,\beta}$, then
\[
\|g_h'\|_{h,\beta}^2 -|g_h(0)|^2  \geq  m h^{1-2\rho} \|g_h\|_{h,\beta}^2\,.
\]
\end{lemma}
\begin{proof}
Let $u^{\rm gs}_{h,\beta}\in\mathfrak D_h^N$ be the normalized (in $X_{h,\beta}$) ground state of the operator $\mathcal H_{h,\beta}^N$:
\begin{equation}\label{eq:3.18}
\mathcal H_{h,\beta}u^{\rm gs}_{h,\beta}=
\lambda_{1.h}^{N}(\beta) u^{\rm gs}_{h,\beta}\,.
\end{equation} 
By the min-max principle,  \eqref{eq:2nd-ev-1Dw}  and  Lemma~\ref{lem:1DL-N}, if $\mathfrak f$ belongs to the form domain of $\mathcal H_{h,\beta}$ and satisfies $\langle \mathfrak f,  u^{\rm gs}_{h,\beta}\rangle_{h,\beta}=0$, then
\begin{equation}\label{eq:wh-dec*}
\langle \mathcal H_{h,\beta} \mathfrak f,\mathfrak f\rangle _{h,\beta}^2\geq \lambda_{2.h}^{N}(\beta)\|\mathfrak f\|^2_{h,\beta}\geq \big(\lambda_{1.h}^{N}(\beta)+c\, h^{1-2\rho}\big)\|\mathfrak f\|^2\,,
\end{equation}
where $c$ is a positive constant.

Now consider a function $g_h\in H^1(0,T)$ such that $\langle g_h,v_h\rangle_{h,\beta}=0$. We  decompose $v_h$ and $g_h$ as follows
\begin{equation}\label{eq:wh-dec}
v_h=\alpha_h u^{\rm gs}_{h,\beta}+f_h~{\rm and}~g_h=\gamma_h u^{\rm gs}_h+e_h\,,
\end{equation}
with
\begin{multline}\label{eq:wh-dec-g}
\alpha_h=\langle v_h,u^{\rm gs}_{h,\beta}\rangle_{h,\beta} ,~\gamma_h=\langle g_h,u^{\rm gs}_h\rangle_{h,\beta}~{\rm and}\\ \langle f_h,  u^{\rm gs}_{h,\beta}\rangle_{h,\beta}=\langle e_h,u^{\rm gs}_{h,\beta}\rangle_{h,\beta}=0\,. 
\end{multline}
We infer from \eqref{eq:1Dev1}, \eqref{eq:1trial-state**}  and \eqref{eq:3.18} that
\[\big\| \big(\mathcal H_{h,\beta}-\lambda_{1.h}^{N}(\beta) \big)f_h\big\|_{h,\beta} =\mathcal O\big(h^{\frac32}\big)  \,.\]
Consequently
\[
q_{h,\beta}(f_h):=\langle \big(\mathcal H_{h,\beta}-\lambda_{1.h}^{N}(\beta) \big)f_h,f_h\rangle_{h,\beta}=\mathcal O\big(h^{\frac32}\big) \|f_h\|_{h,\beta}\,,\]
and by \eqref{eq:wh-dec*},
\[ q_{h,\beta}(f_h)
 \geq c\, h^{1-2\rho}\|f_h\|^2\,. \]
  Eventually we get that 
 \[\big(1-|\alpha_h|^2\big)^{1/2}=\|f_h\|_{h,\beta}=\mathcal O(h^{\frac12+2\rho})\,,\]
 where $\alpha_h$ is introduced in \eqref{eq:wh-dec}. 
 
We return to the function $ g_h$  in \eqref{eq:wh-dec}. Since $e_h\bot u_{h,\beta}^{\rm gs}$,  we get by \eqref{eq:wh-dec*},
 \begin{equation}\label{eq:wh-dec**}
\|g_h'\|_{h,\beta}^2-|g_h(0)|^2=q_{h,\beta}(e_h)  
 \geq c\, h^{1-2\rho}\|e_h\|^2\,.
 \end{equation}
 Since $\langle g_h, v_h\rangle_{h,\beta}=0$,  we get from \eqref{eq:wh-dec},
 \[ \gamma_h\overline{\alpha_h}+\langle e_h,f_h\rangle_{h,\beta}=0 \]
 which yields that
 \[|\gamma_h|\leq \frac1{|\alpha_h|}\|e_h\|\|f_h\|=\mathcal O(h^{\frac12+2\rho})\|e_h\|_{h,\beta}\]
 and consequently
 \[\|g_h\|^2_{h,\beta}=|\gamma_h|^2+\|e_h\|^2_{h,\beta}=\big(1+\mathcal O(h^{1+4\rho})\big)\|e_h\|^2_{h,\beta}\,.\]
 Inserting this into \eqref{eq:wh-dec**}, we finish the proof of Lemma~\ref{lem:gs=app}.
 \end{proof}

\section{The effective operator}\label{sec:eff-op}

\subsection{The operator near  the boundary}~\\
Assume that $\Omega$ is simply connected, hence $\Gamma$ consists of a single connected component. In the case of a multiply connected domain, with $\Gamma$ having a finite number of connected components, we can do the constructions below in each connected component of $\Gamma$.

We introduce the coordinates $(s,t)$ valid in a tubular neighborhood of the boundary, $\Omega_\varepsilon:=\{x\in\Omega,~{\rm dist}(x,\partial\Omega)<\varepsilon\}$, and defined as follows: $t(x)={\rm dist}(x,\Gamma)$ measures the transversal distance to $\Gamma-\partial\Omega$, and $s(x)\in[-L,L)$ measures the (arc-length) tangential distance along $\Gamma$, with $2L=|\Gamma|$ is the length of the boundary. More precisely, we denote
by $[-L,L[\ni s\mapsto M(s)$  the arc-length parameterization of $\Gamma$ oriented counter-clock wise and consider the transformation
\[\Phi:(s,t)\mapsto M(s)-t\nu(s) \]
where $\nu(s)$ is the unit outward normal of $\partial\Omega$.   

The $L^2$-norm of $u$ in $\Omega_\varepsilon$ is 
\[ \|u\|_{L^2(\Omega_\varepsilon)}^2=\int_{-L}^L\int_0^\varepsilon|u(s,t)|^2a(s,t)dtds\]
and the operator $\mathcal T_h$ is expressed as follows
\[\mathcal T_h=-a^{-1}\partial_t(a\partial_t)+a^{-1}\partial_s(a^{-1}\partial_s)\]
where
\[a(s,t)=1-t\kappa(s)\]
and $\kappa(s)$ is the curvature of $\Gamma$ at the point $M(s)$.

 For every $c\in \R$, let $\mathcal L_h^c$ denote the operator (on $\R/2L\mathbb Z$) 
\[ \mathcal L_h^c=-(h^{1/2}+ch^{3/4}) \frac{d^2}{ds^2}-\kappa(s)-\frac12h^{1/2}\big(\kappa(s)\big)^2+ch^{7/8}\,,\]
with domain
\[\mathfrak D=\{u\in H^2(]-L,L[)~:~u(-L)=u(L)\,\&\,u'(-L)=u(L)\}\,.\]
For a self-adjoint semi-bounded operator $\mathcal P$, we denote by $(\lambda_n(\mathcal P))_{n\geq 1}$ the sequence of min-max eigenvalues. For all   $h>0$ and $\epsilon\in\R$, we introduce the following subset of $\mathbb N$
\begin{equation}\label{eq:Ih}
I_h^\epsilon=\{k \geq 1~:~\lambda_k(\mathcal T_h)<\epsilon h\}\,.
\end{equation}

\begin{theorem}\label{thm:main}
Given $0\leq \epsilon<\lambda^N_2(\Omega)$,
there exist positive constants $\mathfrak c,h_0$, such that, for all $h\in(0,h_0]$ and $n\in I_h^\epsilon$,
\begin{equation}\label{eq:main*}
h^{3/2}\min\big( \lambda_n(\mathcal L_h^{-\mathfrak c}) ,h^{-1/2}\big)\leq  \lambda_n(\mathcal T_h)+h\leq h^{3/2}\min\big(\lambda_n(\mathcal L_h^{\mathfrak c}),\epsilon h^{1/2}\big) \,.
\end{equation}
 In particular, for $\lambda_n(\mathcal T_h)<0$, we have,
\begin{equation}\label{eq:main}
 h^{3/2} \lambda_n(\mathcal L_h^{-\mathfrak c})  \leq  \lambda_n(\mathcal T_h)+h\leq h^{3/2}\min\big(\lambda_n(\mathcal L_h^{\mathfrak c}),0\big) \,.
\end{equation} 
\end{theorem}

\begin{remarks}\label{rem:main}\rm~\\
{\bf 1.} Note that for $\epsilon<0$ a stronger result is proven in  \cite{HKR}.\medskip\\
{\bf 2.}  Let  $\lambda_1^D(\Omega)$  be the first  eigenvalue of the Dirichlet Laplacian  on $\Omega$. It follows from \cite{S, W} that $\lambda_2^N(\Omega)<\lambda_1^D(\Omega)$ (see also \cite[Eq.~(2.2)]{A}). The upper bound in \eqref{eq:main*} actually holds for $\epsilon<\lambda_1^D(\Omega)$.\medskip\\
{\bf 3.}  A comparison similar to the one in Theorem~\ref{thm:main}  has been proved in \cite{HKR} when $n\in I_h^{-\epsilon}:=\{k\geq 1\,:\, \lambda_k(\mathcal T_h)<-\epsilon h\}$ with $0<\epsilon<1$ a fixed constant. More precisely, the effective operator in \cite{HKR} is of the form
\[ -\big(h^{1/2}+h b(s)\big)\frac{d^2}{ds^2}-\kappa(s)\,,\]
 with $b(s)=\mathcal O(1)$ uniformly w.r.t. $s$. 
Our result extends that in \cite{HKR}  all the way up to $\epsilon=0$, but with a worse remainder term for the coefficient of $\frac{d^2}{ds^2}$,  in order to consider   all the non-positive eigenvalues.\medskip\\
{\bf 4.} Note that for the realization of $-\partial_s^2$ on $\mathbb R/2L \mathbb Z$, the spectrum is
\[\{ \pi^2 L^{-2}( n-1)^2,~n\geq 1\}\]
with the first eigenvalue being simple and the others being of multiplicity $2$, hence
\[\lambda_1(-\partial_s^2)= 0\]
and
\[ \lambda_{2k}(-\partial_s^2)=\lambda_{2k+1}(-\partial_s^2) =\pi^2 L^{-2}k^2,~k= 1,2,\cdots.\]
Theorem~\ref{thm:main} then yields  the existence of $\mathfrak c >0$  and $h_0>0$ such that for $h\in (0,h_0]$ and $n\in\{2k,2k+1\}$ with $\lambda_n(\mathcal T_h) <\epsilon h^2$, we have 
\[ h^{-3/2}\lambda_n(\mathcal T_h)\leq -h^{-1/2}+\frac{\pi^2 k^2}{L^2}\big(1+\mathfrak ch^{1/4}\big)h^{1/2} -\kappa_{\min} + M_+h^{1/2}+\mathfrak ch^{7/8}\,,\]
and
\[ h^{-3/2}\lambda_n(\mathcal T_h)\geq -h^{-1/2}+\frac{\pi^2 k^2}{L^2}\big(1-\mathfrak ch^{1/4}\big)h^{1/2}-\kappa_{\max} + M_-h^{1/2}-\mathfrak ch^{7/8}\,,\]
 where
\begin{equation*}
\begin{array}{ll}
\kappa_{\min}=\min_{s\in[-L,L)}\kappa(s)\,, &  
 \kappa_{\max}=\max_{s\in[-L,L)}\kappa(s),\\
M_-=-\max_{s\in[-L,L)}|\kappa(s)|^2\,, & M_+=-\min_{s\in[-L,L)}|\kappa(s)|^2\,.\end{array}
\end{equation*}
These estimates  yield Theorem~\ref{thm:main**}. \medskip\\
{\bf 5.} For a positive integer $k=k(h)\gg h^{-1/4}$ satisfying 
\[ k^2\leq (1+\mathfrak ch^{1/4})^{-1}\Big(1+h^{1/2}\kappa_{\max}+\frac12h\kappa_{\min}^2-\mathfrak ch^{7/8}\Big)h^{-1}\,,\] we get 
\[\lambda_{2k}(\mathcal T_h)\sim -h+\pi^2L^{^2}(2k-1)^2h^2 ~\&~ \lambda_{2k+1}(\mathcal T_h)-\lambda_{2k}(\mathcal T_h)=\mathcal O(h^{3/2})\,.
\]
\end{remarks}

\subsection{Decomposition of $L^2(\Omega)$}

Let $\rho\in(\frac13,\frac12)$  and consider the domain $\Omega_{h^\rho}$ defined by \eqref{eq:Om-ep}.

We decompose the Hilbert space $L^2(\Omega)$ as $L^2(\Omega_{h^\rho})\oplus L^2(\Omega\setminus\overline\Omega_{h^\rho})$. We will decompose further the space $L^2(\Omega_{h^\rho})$ by considering the orthogonal projection on the  function 
\[u^{\rm tran}_h(s,t)= c_hh^{-1/4}\chi\big(h^{-\rho}t\big) u^{\rm app}_{h,\kappa(s)}(h^{-1/2}t)\]
where $u_{h,\kappa(s)}^{\rm app}$ is the function defined by \eqref{eq:1trial-state}, and
$\chi\in C_c^\infty(\R)$ satisfies $0\leq \chi\leq 1$, ${\rm supp}\,\chi\subset (-1,1)$ and  $\chi_{/_{[-\frac12,\frac12]}}=1$. 
The coefficient $c_h$ is determined by  $\|u^{\rm tran}_h\|_{L^2(\Omega)}=1$ and satisfies $c_h=1+\mathcal O(h^\infty)$.

Note that $ u^{\rm tran}_h\in D(\mathcal T_h)$ and by \eqref{eq:1trial-state*}
\[   \Big(-a^{-1}\partial_t(a\partial_t)-\lambda_h(s)\Big)u^{\rm tran}_h(s,t) =\mathcal O(h^{\frac32+2\rho})\,,\]
with
\begin{equation}\label{eq:lambda(s)}
\lambda_h(s)=-h-h^{3/2}\kappa(s)-\frac12h^2\kappa(s)^2\,.
\end{equation}
We introduce the projections in the space $L^2(\Omega_{h^\rho})$,
\[\Pi_s \psi = \langle \psi,u^{\rm tran}_h\rangle u^{\rm tran}_h\quad{\rm and}\quad \Pi_s^\bot\psi=\psi_{/_{\Omega_{h^\rho}}}-\Pi_s\psi\,,  \]
 and the isometry
\begin{equation}\label{eq:isom-L2-Om-h-rho}
\psi\ni L^2(\Omega_{h^\rho})\mapsto  (\Pi_s\psi,\Pi_s^\bot\psi)\in V_h\oplus W_h 
\end{equation}
where
\begin{equation*}
\begin{array}{ll}
V_h&=L^2\big([-L,L)\big)\otimes\{ {\rm span} (  u^{\rm tran}_h) \} \\ 
&=\{v\in L^2(\Omega_{h^\rho}),~\exists\,k\in L^2([-L,L)),~v(s,t)=k(s)u^{\rm tran}_h(s,t)\}
\end{array}
\end{equation*}
and
\[W_h=\{v\in L^2(\Omega_{h^\rho}),~\int_{0}^{h^\rho} v(s,t) u^{\rm tran}_h(s,t)\,(1-t\kappa(s))dt=0\}\,.\]
 Using \eqref{eq:isom-L2-Om-h-rho} and the decomposition of $L^2(\Omega)$ as $L^2(\Omega_{h^\rho})\oplus L^2(\Omega\setminus\overline\Omega_{h^\rho})$, we construct the following isometry 
\begin{equation}\label{eq:isom-L2-Om-h-rho}
\psi\ni L^2(\Omega)\mapsto \chi_\psi:= (\Pi_s\psi,\Pi_s^\bot\psi,\psi_{/_{\Omega\setminus\overline{\Omega_{h^\rho}}}})\in V_h\oplus W_h\oplus L^2(\Omega\setminus\overline\Omega_{h^\rho} )
\end{equation}\Bk
Note that
\[\|\psi\|^2_{L^2(\Omega)}=\|\chi_\psi\|^2=\int_{-L}^L |k_\psi(s)|^2ds+\int_{\Omega_{h^\rho}} |\Pi_s^\bot\psi|^2dx+\int_{\Omega\setminus\overline{\Omega_{h^\rho}}}|\psi|^2dx\,,\]
where
\begin{equation}\label{eq:k-psi}
k_\psi(s):= \langle \psi, u^{\rm tran}_h\rangle=\int_{0}^{h^\rho} \psi(s,t)\,u^{\rm tran}_h(s,t)\,(1-t\kappa(s))dt\,.
\end{equation}
\subsection{Decomposition of the quadratic form}
We examine the quadratic form
\begin{equation}\label{eq:qf}
\begin{aligned}
q_h^\Omega(\psi)&:=h^2\int_\Omega |\nabla \psi|^2dx-h^{3/2}\int_{\partial\Omega}|\psi|^2ds(x)\\
&=q_h^{\Omega_{h^\rho}}(\psi)+q_{h,\rho}^{\rm int}(\psi)\\
\end{aligned}
\end{equation}
where
\[q_{h,\rho}^{\rm int}(\psi)=\int_{\Omega\setminus\overline\Omega_{h^\rho}}|\nabla\psi|^2\,dx\,.\]
Working in the $(s,t)$ coordinates, we express the quadratic form $q_h^{\Omega_{h^\rho}}(\psi)$ as follows
\[q_h^{\Omega_{h^\rho}}(\psi)=h^2\int_{-L}^L\left(\int_0^{h^\rho}\Big( |\partial_t\psi|^2+a^{-2}|\partial_s\psi|^2\Big)a dt-h^{-1/2}|\psi(s,t=0)|^2 \right)\,.\]
Freezing the $s$-variable, the $\Pi_s$ is an orthogonal projection in the weighted Hilbert space $L^2\big((0,h^\rho);a(s,t)dt\big)$; consequently,
\begin{align*}
q^{\rm tran}_h(\psi)&:=h^2\int_0^{h^\rho}|\partial_t\psi|^2(1-t\kappa(s))dt-h^{3/2}|\psi(s,t=0)|^2\\
&=q^{\rm tran}_h(\Pi_s\psi)+q^{\rm tran}_h(\Pi_s^\bot\psi)
\end{align*}
and
\[
\int_0^{h^\rho}|\partial_s\psi|^2(1-t\kappa(s))dt=\int_0^{h^\rho}\Big(|\Pi_s\partial_s\psi|^2+|\Pi_s^\bot\partial_s\psi|^2\Big)(1-t\kappa(s))dt
\]
We have (see \eqref{eq:lambda(s)})
\begin{equation}\label{eq:qf-tran}
\begin{aligned}
q^{\rm tran}_h(\Pi_s\psi)&=\big(\lambda_h(s)+\mathcal O(h^{\frac32+2\rho})\big)\int_{0}^{h^\rho}|\Pi_s\psi|^2(1-t\kappa(s))dt\\
&=\big(\lambda_h(s)+\mathcal O(h^{\frac32+2\rho})\big)|k_\psi(s)|^2\,,
\end{aligned}
\end{equation}
\begin{equation}\label{eq:qf>0}
q^{\rm tran}_h(\Pi_s^\bot\psi) \gtrsim h^{2-2\rho}\,,
\end{equation}
and, setting $\mathfrak K =8\|\kappa\|_\infty$,
\[1+2t\mathfrak K\leq a^{-2}\leq 1+t\mathfrak K\,.\]
Therefore, we end up with the following upper bound of the quadratic form
\begin{multline}\label{eq:qf-ub-2D}
q_h^{\Omega_{h^\rho}}(\psi)
\leq \\
h^2\int_{-L}^L\left(\big(\lambda_h(s)+\mathcal O(h^{\frac32+2\rho})\big) |k_\psi(s)|^2+\int_0^{h^\rho}(1+\mathfrak Kt) |\Pi_s\partial_s\psi|^2adt\right)ds\\
+
h^2\int_{-L}^L\int_0^{h^\rho} \Big(|\partial_t\Pi_s^\bot\psi|^2+(1+\mathfrak Kt) |\Pi_s^\bot\partial_s\psi|^2\Big)a dtds\\
-h^{3/2}\int_{-L}^L|\Pi^\bot_s\psi(s,t=0)|^2ds\,.
\end{multline}
The same argument yields the following lower bound
\begin{multline}\label{eq:qf-lb-2D}
q_h^{\Omega_{h^\rho}}(\psi)
\geq \\
h^2\int_{-L}^L\left(\big(\lambda_h(s)+\mathcal O(h^{\frac32+2\rho})\big) |k_\psi(s)|^2+\int_0^{h^\rho}(1-\mathfrak Kt) |\Pi_s\partial_s\psi|^2adt\right)ds\\
+
h^2\int_{-L}^L\int_0^{h^\rho} \Big(|\partial_t\Pi_s^\bot\psi|^2+(1-\mathfrak Kt) |\Pi_s^\bot\partial_s\psi|^2\Big)a dtds\\
-h^{3/2}\int_{-L}^L|\Pi^\bot_s\psi(s,t=0)|^2ds\,.
\end{multline}
Let us now handle the term $|\Pi_s\partial_s\psi|$. Let us introduce $u=\partial_s\psi$. It is easy to check the following identities,
\begin{align*}
\partial_s\Pi_s\psi 
&=\partial_s\Big( k_\psi(s) u^{\rm tran}_h\Big)\\
&= k'_\psi(s) u^{\rm tran}_{h}+k_\psi(s) \partial_s u^{\rm tran}_h \\
&=\big(k_u -\kappa'(s)k_{t\psi}+\langle \psi,\partial_s u^{\rm tran}_h\rangle\big)u^{\rm tran}_{h}+k_\psi(s) \partial_s u^{\rm tran}_h\,.
\end{align*}
Therefore,
\[ \Pi_s\partial_s\psi=k'_\psi(s) u^{\rm tran}_{h} + \underset{:=w_\psi(s,t) }{\underbrace{k_\psi(s) \partial_s u^{\rm tran}_h +\big(\kappa'(s)k_{t\psi}(s) -\langle \psi,\partial_s u^{\rm tran}_h\rangle\big)u^{\rm tran}_{h}}}\,. \]
Note that if we perform the change of variable, $t=h^{1/2}\tau$, we can write
\[\int_0^{h^\rho} t |u^{\rm tran}_{h}(s,t)|^2dt=\mathcal O(h^{1/2})\]
uniformly with respect to $s$. In a similar manner, we can check that
\[ \int_0^{h^\rho}|\partial_s u^{\rm tran}_h|^2dt=\mathcal O(h)\,.\] 
Consequently, if we introduce the norms
\[N^\pm(f)=\left(\int_{-L}^L\int_0^{h^\rho} (1\pm\mathfrak Kt) |f|^2adtds\right)^{1/2}\,,\] we get that
\[ N^\pm (k_\psi u^{\rm tran}_h)^2=\big(1+\mathcal O(h^{1/2})\big)\int_{-L}^L |k_\psi'(s)|^2ds\]
and
\[ N^\pm (w_\psi)^2=\mathcal O(h^{1/2})\int_{-L}^L\int_0^{h^\rho}|\psi|^2dsdt=\mathcal O\big(h^{1/2}\|\psi\|^2_{L^2(\Omega_{h^{\rho}})}\big)\,. \]
Armed with the foregoing estimates, and Cauchy's inequality, we write, for all $\eta\in(0,1)$,
\[ N^\pm(\Pi_s\partial_s\psi)\geq (1-\eta)N^\pm (k_\psi u^{\rm tran}_h)^2-\eta^{-1}N^\pm (w_\psi)^2  \]
and
\[ N^\pm(\Pi_s\partial_s\psi)\leq (1+\eta)N^\pm (k_\psi u^{\rm tran}_h)^2+(1+\eta^{-1})N^\pm (w_\psi)^2\,.  \]
Choosing $\eta=h^{1/4}$, we eventually get estimates for the energy of $\Pi_s\partial_s\psi$ as follows
\begin{equation}\label{eq:en-d-Pi-s}
\begin{aligned}
 &\left| \int_{-L}^L\int_0^{h^\rho}\int_0^{h^\rho}(1\pm\mathfrak Kt) |\Pi_s\partial_s\psi|^2adtds -\int_{-L}^L |k'_\psi(s)|^2ds\right|\\
&\leq M h^{1/4} \Big(\int_{-L}^L |k'_\psi(s)|^2ds+  \|\psi\|_{L^2(\Omega_{h^\rho})}^2\Big)\\
&=M h^{1/4} \Big(\int_{-L}^L |k'_\psi(s)|^2ds+  \underset{=\int_{[-L,L)}|k_\psi(s)|^2ds}{\underbrace{\|\Pi_s\psi\|^2_{L^2(\Omega_{h^\rho})}}}+\|\Pi_s^\bot\psi\|^2_{L^2(\Omega_{h^\rho})}\Big)
\end{aligned}
\end{equation}
for $h\in(0,h_0] $, where  $M,h_0$ are positive constants. 

\subsection{Comparison of eigenvalues}

\subsubsection{Upper bounds}
 Consider the self-adjoint operator $\mathcal T_h^{K_h}$ in $L^2(\Omega)$, defined by the quadratic form
\begin{equation}\label{eq:qf-Kh}
K_h\ni u\mapsto h^2\int_\Omega |\nabla \psi|^2dx-h^{3/2}\int_{\partial\Omega}|u|^2ds(x)
\end{equation}
where the form domain $K_h$ consists of functions with zero trace on the boundary of $\Omega\setminus\overline\Omega_{h^\rho}$, i.e.
\[
K_h=\{ v\in H^1(\Omega)~:~ v=0~{\rm for~}{\rm dist}(x,\Gamma)=h^\rho\}
\]
By the min-max principle and comparison of the form domains, for all $n\geq 1$,
\begin{equation}\label{eq:Dirch-brak}
\lambda_n(\mathcal T_h)\leq \lambda_n(\mathcal T_h^{K_h})\,.
\end{equation}
For all $\psi\in K_h$, we investigate the quadratic form
\[
q_h^\Omega(\psi)=h^2\int_\Omega |\nabla \psi|^2dx-h^{3/2}\int_{\partial\Omega}|u|^2ds(x)\leq \mathfrak q_h(\chi_\psi)
\]
where 
\[\mathfrak q_h(\chi_\psi):=\mathfrak q_{h,1}(k_\psi)+\mathfrak q_{h,2}( f_\psi)+\mathfrak q_{h,3}(u_\psi)\]
and
\[\chi_\psi=(k_\psi u^{\rm tran}_h, f_\psi, u_\psi)\,.\]
The quadratic forms $\mathfrak q_{h,i}$ are defined as follows
\begin{align*}
&\mathfrak q_{h,1}(k_\psi)=\int_{-L}^L \Big( \big(\lambda_h(s)+\mathcal O(h^{\frac32+2\rho})\big)|k_\psi(s)|^2+\big(h^2+\mathcal O(h^{\frac94})\big)|k_\psi'(s)|^2\Big)ds\,,\\
&\mathfrak q_{h,2}(f_\psi)=h^2\int_{-L}^L\left(\int_0^{h^\rho}\Big( |\partial_tf_\psi|^2+\mathcal O(h^{\frac14})|f_\psi|^2\Big)adt-h^{-\frac12}|f_\psi(s,0)|^2\right)ds\\
&\mathfrak q_{h,3}(u_\psi)=h^2\int_{\Omega_{h^\rho}}|\nabla u_\psi|^2dx\geq h^2\lambda_1^D(\Omega)\|u_\psi\|_{L^2(\Omega_{h^\rho})}^2\,.
\end{align*}
For all $i\in\{1,2,3\}$, let $\mathfrak L_{h,i}$ be the operator defined by the quadratic form $\mathfrak q_{h,i}$. By the min-max principle,
\[\lambda_n(\mathcal T_h^{K_h})\leq \lambda_n\big(\oplus_{i=1}^3 \mathfrak L_{h,i}\big)\,.\]
We insert this into \eqref{eq:Dirch-brak} and choose $\rho=\frac7{16}\in(\frac13,\frac12)$. 
Note that $\lambda_n(\mathfrak L_{h,3})> 0$  for all $n\geq 1$. 

Since $f_\psi\bot u_h^{\rm tran}$ in $L^2\big((0,h^{\rho});(1-t\kappa(s))dt)$, we get by Lemma~\ref{lem:gs=app} and our choice of $\rho=\frac7{16}$ that $\lambda_n(\mathfrak L_{h,2})\gtrsim h^{2-2\rho}=h^{9/8}>0$ for all $n\geq 1$.

Thus, we end up with
\begin{equation}\label{eq:Diric-brak-*}
\forall\,n\in I_h,~\lambda_n(\mathcal T_h)\leq -h+h^{3/2}\min\Big(\lambda_n(\mathcal L_h^+) ,0\Big)\,,
\end{equation} 
where, for some constant $\mathfrak c_+>0$, $\mathcal L_h^+$ is the operator acting on $L^2\big([-L,L)\big)$ as follows,
\begin{equation}\label{eq:op-Lh+}
\mathcal L_h^+=-h^{1/2}(1+\mathfrak c_+h^{1/4})\frac{d^2}{ds^2}-\kappa(s)-\frac{1}2h^{1/2}\kappa(s)^2+\mathfrak c_+h^{7/8}\,.
\end{equation}

If we consider the eigenvalues of $\mathcal T_h$ below the energy level $\epsilon h^2$, with $\epsilon<\lambda_1^D(\Omega)$, we still get
\[\lambda_n(\mathcal T_h)\leq -h+h^{3/2}\min\Big(\lambda_n(\mathcal L_h^+) ,\epsilon h^2\Big)\]

\subsubsection{Lower bounds}

For all $\psi\in H^1(\Omega)$, we write the lower bound
\[
q_h^\Omega(\psi)=h^2\int_\Omega |\nabla \psi|^2dx-h^{3/2}\int_{\partial\Omega}|u|^2ds(x)\geq \mathfrak p_h(\chi_\psi)
\]
where 
\[\mathfrak p_h(\chi_\psi):=\mathfrak p_{h,1}(k_\psi)+\mathfrak p_{h,2}( f_\psi)+\mathfrak p_{h,3}(u_\psi)\]
and
\[\chi_\psi=(k_\psi u^{\rm tran}_h, f_\psi, u_\psi)\,.\]
The quadratic forms $\mathfrak p_{h,i}$ are defined as follows
\begin{align*}
&\mathfrak p_{h,1}(k_\psi)=\int_{-L}^L \Big( \big(\lambda_h(s)-\mathcal O(h^{\frac32+2\rho})\big)|k_\psi(s)|^2+\big(h^2-\mathcal O(h^{\frac94})\big)|k_\psi'(s)|^2\Big)ds\,,\\
&\mathfrak p_{h,2}(f_\psi)=h^2\int_{-L}^L\left(\int_0^{h^\rho}\Big( |\partial_tf_\psi|^2-\mathcal O(h^{\frac14})|f_\psi|^2\Big)adt-h^{-\frac12}|f_\psi(s,0)|^2\right)ds\\
&\mathfrak p_{h,3}(u_\psi)=h^2\int_{\Omega_{h^\rho}}|\nabla u_\psi|^2dx\geq h^2\lambda_1^N(\Omega_{h^\rho})\|u_\psi\|_{L^2(\Omega_{h^\rho})}^2=0\,.
\end{align*}
For all $i\in\{1,2,3\}$, let $\mathfrak l_{h,i}$ be the operator defined by the quadratic form $\mathfrak p_{h,i}$. By the min-max principle,
\[\lambda_n(\mathcal T_h)\geq \lambda_n\big(\oplus_{i=1}^3 \mathfrak l_{h,i}\big)\,,\]
with $\lambda_1(\mathfrak l_{h,3})\geq  0$
 and $\lambda_n(\mathfrak l_{h,3})>  0$ for all $n\geq 2$. We choose $\rho=\frac7{16}$ and observe that, by Lemma~\ref{lem:gs=app}, $\lambda_n(\mathfrak l_{h,2})\gtrsim h^{9/8}>  0$ for all $n\geq 1$. Therefore, for some constant $\mathfrak c_->0$ and $\mathcal L_h^{-}$  the operator (acting on $L^2\big([-L,L)\big)$)
\begin{equation}\label{eq:op-Lh-}
\mathcal L_h^{-}=-h^{1/2}(1-\mathfrak c_-h^{1/4})\frac{d^2}{ds^2}-\kappa(s)-\frac{1}2h^{1/2}\kappa(s)^2-\mathfrak c_-h^{7/8}\,,
\end{equation}
we get
\begin{equation}\label{eq:Diric-brak-**}
\forall\,n\in I_h^0,~\lambda_n(\mathcal T_h)\geq -h+h^{3/2}\lambda_n(\mathcal L_h^{-})\,.
\end{equation} 
When dealing with the eigenvalues of $\mathcal T_h$ below  $\epsilon h^2$, with $\epsilon<\lambda_2^N(\Omega)$, we still get
\[\lambda_n(\mathcal T_h)\geq -h+h^{3/2}\min\big(\lambda_n(\mathcal L_h^{-}),h^{-1/2}\big)\,,\]
because 
\[\min\big(-h+h^{3/2}\lambda_n(\mathcal L_h^{-}),0\big)=-h+\min\big(h^{3/2}\lambda_n(\mathcal L_h^{-}),h\big)\,.  \]

\begin{remark}\label{rem:e.v.<ep}~\rm Consider $\epsilon\in(0,\lambda_2^N(\Omega))$. Since $\lambda_1(\mathfrak l_{h,3})=0$ is a simple eigenvalue, the min-max principle allows us to extend \eqref{eq:Diric-brak-**} as follows. 
Set $N_*(h)=\max\{n\geq 1,\,-h+h^{3/2}\lambda_n(\mathcal L_h^{-})<0\}$. Then, for $h$ small enough, we have
\[\forall\,n\in I_h^\epsilon\cap[2+N_*(h),+\infty),~\lambda_n(\mathcal T_h)\geq -h+h^{3/2}\lambda_n(\mathcal L_h^{-})\,. \] 
\end{remark}

\subsection*{Acknowledgements}~
The authors would like to thank  Gerd Grubb,  Thierry Daud\'e and Fran\c cois Nicoleau for  helpful discussions.  The first author was inspired by the very interesting talks 
 proposed at the seminar ``Spectral geometry in the clouds" organized  by A. Girouard and J. Lagac\'e and initially due to this terrible COVID period.  The second author is supported  by the Lebanese University within the project \textit{``Analytical and numerical aspects of the Ginzburg Landau model''}.

\end{document}